\documentclass[11pt, reqno]{amsart}
\usepackage{mathrsfs}
\usepackage[active]{srcltx}
\usepackage{mathrsfs}
\usepackage{longtable}
\usepackage[reqno]{amsmath}

\usepackage{xcolor}
\definecolor{bluecite}{HTML}{0875b7}
\usepackage[unicode=true,
bookmarksopen={true},
pdffitwindow=true,
colorlinks=true,
linkcolor=bluecite,
citecolor=bluecite,
urlcolor=bluecite,
hyperfootnotes=false,
pdfstartview={FitH},
pdfpagemode= UseNone]{hyperref}

\newcommand{\ds}{\displaystyle}

\usepackage[left=2.5cm,right=2.50cm,top=2.6cm,bottom=2.5cm]{geometry}

\newtheorem{example}{ Example}[section]
\newtheorem{proposition}{Proposition}[section]
\newtheorem{theorem}{Theorem}[section]
\newtheorem{lemma}{Lemma}[section]

\newtheorem{remark}{Remark}[section]
\numberwithin{equation}{section}

\usepackage{bbm}  
\usepackage{dsfont}

\begin{document}
	
	\author{Alexandru Krist\'aly}
	\address{Institute of Applied Mathematics\\
		\'Obuda University\\
		1034 Budapest, Hungary \& Department of Economics\\
		Babe\c s-Bolyai University\\
		400591 Cluj-Napoca, Romania}
	\email{alex.kristaly@econ.ubbcluj.ro; kristaly.alexandru@nik.uni-obuda.hu}
	
	\author{Andrea Mondino}
	\address{
		Mathematical Institute\\
		University of Oxford\\
		Oxford OX1 3LB\\
		England	}
	\email{andrea.mondino@maths.ox.ac.uk}
	
	\renewcommand{\subjclassname}{%
		\textup{2020} Mathematics Subject Classification}
	\thanks{A.\ Krist\'aly is  supported by the
		Excellence Researcher Program \'OE-KP-2-2022 of \'Obuda University, Hungary}
	\thanks{A.\;Mondino acknowledges support from the European Research Council (ERC) under the European Union's Horizon 2020 research and innovation programme, grant agreement No.~802689 ``CURVATURE''
	}

	\title[Clamped plates on ${\sf RCD}(0,N)$ spaces]
	{Principal frequency of clamped plates on ${\sf RCD}(0,N)$ spaces:\\ sharpness, rigidity and stability
	}

	\begin{abstract}
		We study fine properties of the principal frequency of clamped plates in the (possibly  singular) setting of metric measure spaces verifying the ${\sf RCD}(0,N)$ condition, i.e.,  infinitesimally Hilbertian spaces with non-negative Ricci curvature  and dimension bounded above by $N>1$ in the synthetic sense. The initial conjecture  -- an isoperimetric inequality for the principal frequency of clamped plates -- has been formulated in 1877 by Lord Rayleigh in the Euclidean case and solved affirmatively in dimensions 2 and 3 by Ashbaugh and Benguria [\textit{Duke Math.\ J.}, 1995] and Nadirashvili [\textit{Arch.\ Rat.\ Mech.\ Anal.}, 1995]. The  main contribution of the present work is  a new isoperimetric inequality for the principal frequency of clamped plates in ${\sf RCD}(0,N)$ spaces whenever $N$ is close enough to 2 or 3. The inequality contains the so-called \textit{asymptotic volume ratio}, and turns out to be sharp under the subharmonicity of the distance function,  a condition satisfied in  metric measure cones. In addition, rigidity (i.e., equality in the isoperimetric inequality) and stability results are established in terms of the cone structure of the ${\sf RCD}(0,N)$ space as well as the shape of the eigenfunction for the principal frequency, given by means of Bessel functions. These results are new even for  Riemannian manifolds with non-negative Ricci curvature. We discuss examples of both smooth and non-smooth spaces where the results  can be applied.  
	\end{abstract}
	\maketitle
	
	\vspace{0.5cm}
	\section{Introduction}
	
	In 1877, Lord Rayleigh \cite[p.\ 382]{Rayleigh} formulated  an isoperimetric inequality   arising from the theory of sound, by claiming that among domains with a fixed volume, balls should have the lowest principal frequency for  vibrating clamped plates. Given an open bounded domain $\Omega\subset \mathbb R^N$, $N\geq 2,$ and $\Omega^\star\subset \mathbb R^N$ a ball with the same volume as $\Omega$, the aforementioned isoperimetric inequality can be reformulated as  
	\begin{equation}\label{cp-conjecture}
			\Lambda(\Omega)\geq \Lambda(\Omega^\star)= h^4_{\nu}\left(\frac{\omega_N}{{\rm Vol}(\Omega)}\right)^{4/N},
	\end{equation}
with equality if and only if $\Omega$ is isometric to the ball $\Omega^\star$ (up to a set of zero $H^2$-capacity), 
	where the principal frequency (or, fundamental tone) for clamped plates are characterized variationally as  
	\begin{equation}\label{variational-charact--membrane}
		\Lambda(\Omega)=\inf_{u\in W_0^{2,2}(\Omega)\setminus \{0\}}\frac{\displaystyle \int_{\Omega}(\Delta u)^2 {\rm d}x}{\displaystyle \int_{\Omega}u^2 {\rm d}x}.
	\end{equation} 
   Here,  $\nu=\frac{N}{2}-1$, $\omega_N=\pi^{N/2}/\Gamma(1+N/2)$  is the volume of the unit Euclidean ball in $\mathbb R^N$, and $ h_\nu$ is the first positive root of the cross-product of the Bessel function $J_\nu$ and modified Bessel function $I_\nu$
   of first kind, which is their Wronskian  having the explicit form $J_\nu I'_\nu-J'_\nu I_\nu$. 
   In \eqref{variational-charact--membrane},   $W_0^{2,2}(\Omega)$ is the usual Sobolev space over $\Omega,$ i.e., the completion of $C_0^\infty(\Omega)$ with respect to the classical norm of $W^{2,2}(\Omega)$.

	The first notable contribution to Lord Rayleigh's conjecture is due to Szeg\H o \cite{Szego} in the early fifties by proving  \eqref{cp-conjecture} whenever the eigenfunction associated to \eqref{variational-charact--membrane} is of constant sign. However, contrary to the common belief at the time -- based on the one-sign property of eigenfunctions for problems involving second-order operators, like the fixed membrane problem described as the Faber--Krahn inequality --  it turned out that there are examples of eigenfunctions for the clamped plate problem which are sign-changing. Such an example was first reported by Duffin \cite{Duffin} and Coffman, Duffin and Shaffer \cite{CDS} on  a specific annulus (e.g., on  a punctured disk), 	the unexpected phenomenon being explained by the absence of a maximum principle for higher order elliptic operators, as it occurs for the vibrating clamped plate problem. In the early eighties, with the knowledge of the possible occurrence of nodal domains for eigenfunctions in clamped plate problems, a landmark argument has been presented by Talenti \cite{Talenti}. 
	In fact, he decomposed the initial domain $\Omega$ into two parts, corresponding to the positive and negative parts of the first eigenfunction, obtaining  a two-balls minimization problem by using Schwarz symmetrization and the classical isoperimetric inequality for sets in $\mathbb R^n$. In this way,  Talenti proved a weak form of Lord Rayleigh's conjecture, namely, instead of \eqref{cp-conjecture}, he obtained that $
	\Lambda(\Omega)\geq t_N \Lambda(\Omega^\star)$ where $t_N\in [1/2,0.98)$ with $t_N\to 1/2$ as $N\to \infty.$ Building on Talenti's argument, in 1995, Nadirashvili \cite{Nadirashvili} proved \eqref{cp-conjecture} for $N=2$, and in the same time, Ashbaugh and Benguria \cite{A-B-Duke} adapted the argument to prove the conjecture \eqref{cp-conjecture} for $N=2$ and $N=3$, by using peculiar properties of the Bessel functions. In 1996, improving Talenti's estimate in high dimensions, Ashbaugh and Laugesen \cite{A-L-Pisa} provided an asymptotically sharp form of Lord Rayleigh's conjecture, i.e., $
	\Lambda(\Omega)\geq c_N \Lambda(\Omega^\star)$ where $c_N\in [0.89,1)$ with $c_N\to 1$ as $N\to \infty.$ These results are summarized  in the monograph of Gazzola, Grunau and Sweers \cite[Chapter 3]{GGS}. 
 Very recently, in 2024, still improving Talenti's argument,  Leylekian \cite{Leylekian-ARMA} (see also \cite{Leylekian-1})  answered Lord Rayleigh's conjecture in high dimensions requiring  an appropriate   regularity of the domain $\Omega$ as well as a controlled behavior of the critical values for the first eigenfunction in \eqref{variational-charact--membrane}. 
	
	A natural question arose concerning the validity of the analogue of Lord Rayleigh's conjecture on non-euclidean structures. First, in 2016, Chasman and Langford \cite{Chasman-Langford} proved a
	non-sharp isoperimetric inequality for the principal frequencies of clamped plates on Gaussian spaces, where the integral in \eqref{variational-charact--membrane} are defined in terms of the Gaussian measure. In 2020, Krist\'aly \cite{Kristaly-Adv} proved that the principal frequency of  `small' clamped plates in low-dimensional negatively curved Riemannian manifolds (endowed with the canonical measure) is at least as large as the corresponding principal frequency of a geodesic ball of the same volume in the model space of constant curvature (either the hyperbolic or the Euclidean space, depending on the upper bound of the sectional curvature). This result is based on fine properties of the Gaussian hypergeometric functions and on the validity of the Cartan-Hadamard conjecture, which holds in dimensions 2 and 3. Complementing the latter result, in 2022, Krist\'aly \cite{Kristaly-GAFA} proved the analogue of Lord Rayleigh's conjecture 
	on 	Riemannian manifolds with positive Ricci curvature for any clamped plate in 2
	and 3 dimensions, and for sufficiently `large' clamped plates in dimensions beyond
	3, the main tool being the L\'evy--Gromov isoperimetric inequality and fine properties of the Gaussian hypergeometric functions of the type $_2F_1$.

	The aim of the present paper is to gain a deeper insight into the clamped plate  problem in abstract metric measure spaces which could  even contain singularities. Such a geometric framework will be provided by spaces verifying the Riemannian curvature-dimension condition ${\sf RCD}(0,N)$, i.e.,  infinitesimally Hilbertian spaces with non-negative Ricci curvature  and dimension bounded above by $N>1$ in the synthetic sense; for details, see \S \ref{section-2}. These spaces contain not only smooth Riemannian manifolds (possibly, with a non-negative weight function on the volume measure) but also spaces containing singularities arising from Gromov--Hausdorff limits of $N$-dimensional (weighted) Riemannian manifolds with non-negative Ricci curvature, Alexandrov spaces with non-negative curvature, etc.

In order to present our results, let $(X,{\sf d},{\sf m})$ be an ${\sf RCD}(0,N)$ space with $N>1$ not necessarily an integer, and $\Omega\subset X$ be an open bounded domain. The \textit{principal frequency} of the clamped plate on  $\Omega$ is defined as
	\begin{equation}\label{Rayleigh}
		\Lambda_{\sf m}(\Omega)=\inf_{u\in H^{2,2}_0(\Omega)\setminus \{0\}}\frac{\displaystyle\int_\Omega (\Delta u)^2 {\rm d}{\sf m}}{\displaystyle\int_\Omega u^2 {\rm d}{\sf m}},
	\end{equation}
	where $H_0^{2,2}(\Omega)$ is the natural non-smooth counterpart of the usual Sobolev space $W_0^{2,2}(\Omega)$; for details, see \S\ref{section-2}. 
	Let $B_r(x)=\{y\in X: {\sf d}(x,y)<r\}$ be the metric ball with center $x\in X$ and radius $r>0.$ By  Bishop--Gromov monotonicity, valid on ${\sf RCD}(0,N)$ spaces, the \textit{asymptotic volume ratio}
	\begin{equation} \label{eq:defAVR}
	{\sf AVR}_ {\sf m}=\lim_{r\to \infty}\frac{ {\sf m}(B_r(x))}{\omega_Nr^N},
	\end{equation}
	is well-defined, i.e., it is independent of the choice of $x\in X$; hereafter, $\omega_N=\pi^{N/2}/\Gamma(N/2+1)$.  For further use, let $\nu=\frac{N}{2}-1.$
	
	The constant ${\sf AVR}_ {\sf m}$ will play a crucial role in our isoperimetric inequality; indeed, our first main result -- similar  in spirit to \eqref{cp-conjecture}  --  reads as follows:

	\begin{theorem}\label{theorem-main} {\rm (Isoperimetric inequality for clamped plates on ${\sf RCD}(0,N)$ spaces)}
		There exists $\varepsilon_0>0$ with the following property. Let $N\in (2-\varepsilon_0,2+\varepsilon_0)\cup (3-\varepsilon_0,3+\varepsilon_0)$ and let  $(X,{\sf d},{\sf m})$ be an  ${\sf RCD}(0,N)$ metric measure space with   ${\sf AVR}_ {\sf m}>0$.  
		For every open bounded domain 
		 $\Omega\subset X$, it holds that
		\begin{equation}\label{inequality}
			\Lambda_{\sf m}(\Omega)\geq {\sf AVR}_ {\sf m}^\frac{4}{N}\Lambda_0(\Omega^*)\equiv {\sf AVR}_ {\sf m}^\frac{4}{N}h_\nu^4\left(\frac{\omega_N}{{\sf m}(\Omega)}\right)^\frac{4}{N},
		\end{equation}
		where $\Omega^*=[0,R]$ with ${\sf m}(\Omega)=\omega_N R^N$,  and $\Lambda_0(\Omega^*)$ stands for  the principal frequency of $\Omega^*$ in the  $1$-dimensional model space $([0,\infty),{\sf d}_{\rm eu},\sigma_N= N\omega_Nr^{N-1}\mathcal L^1)$. 
			\end{theorem}

		The proof of Theorem \ref{theorem-main} is inspired by  the one of Ashbaugh and Benguria \cite{A-B-Duke}; however, several crucial differences appear that we explain in the sequel by outlining the argument. Let $u\in H_0^{2,2}(\Omega)$ be an eigenfunction to \eqref{Rayleigh}, whose existence is guaranteed by standard embedding arguments (see Proposition \ref{prop-inf-min}). Since we do not know a priori anything about the sign of $u$, we consider the subsets $\Omega_+=\{u>0\}$,  $\Omega_-=\{u<0\}$ and $\Omega_0=\{u=0\}$ of $\Omega.$ Due to the lack of higher regularity of $u$ -- in fact, we only know that $u$ is locally Lipschitz, whose proof requires the dimensional restriction $N<4$ (see Proposition \ref{prop-lipschitz}), -- we shall restrict our analysis to the set $\tilde \Omega=\Omega_-\cup \Omega_+$\footnote{We thank Professor Mark Ashbaugh for suggesting  this idea, which circumvents the regularity of the eigenfunction used in the smooth setting.};
		this observation is in concordance incidentally  with  the fact that the unique continuation of harmonic functions on ${\sf RCD}(0,N)$ spaces generally fails whenever $N\geq 4$, see Deng and Zhao \cite{Deng-Zhao}. By using a suitable symmetrization on $\Omega_-$ and $\Omega_+$, we reduce our variational problem to a coupled  ``two-balls problem"  on the 1-dimensional model space $([0,\infty),{\sf d}_{\rm eu},\sigma_N=N\omega_Nr^{N-1}\mathcal L^1)$. To perform such a strategy, we explore three non-trivial features of ${\sf RCD}(0,N)$ spaces, namely: (a) the co-area formula (see Miranda  \cite{Miranda});  (b) a particular form of the Gauss--Green formula (see Proposition \ref{prop-div-them} based on Bru\'e, Pasqualetto and Semola \cite{BPS-ASNS}); and (c) the sharp isoperimetric inequality which contains the constant ${\sf AVR}_ {\sf m}$  (see Balogh and Krist\'aly \cite{BK}). 
				
		 In contrast to \cite{A-B-Duke} -- where the Euclidean space $\mathbb R^N$ has been considered for $N\in \{2,3\}$ -- we are interested in a wider range of not necessarily integers $N>1$  where the aforementioned ``two-balls problem" can be handled. Due to the lack of regularity, it is not likely that the approach of Leylekian \cite{Leylekian-ARMA} works for large values of $N$. In fact, numerical computations show that one could expect the validity of \eqref{inequality} for every $N\in (1,N_0)$ with $N_0\approx 3.30417$, but serious technical difficulties arise -- based on  peculiar properties of Bessel functions --   which make challenging to obtain an analytical proof. However, the arguments of Ashbaugh and Benguria \cite{A-B-Duke} combined with fine estimates, involving quantities associated to Bessel functions, show  the validity of \eqref{inequality} in  close neighborhoods of the integer dimensions $2$ and $3$, as  claimed in  Theorem \ref{theorem-main}. 
		 
		 Inequality \eqref{inequality} was already established by Krist\'aly \cite[Theorem 1.4]{Kristaly-GAFA} in the smooth setting of $2$- and $3$-dimensional Riemannian manifolds with non-negative Ricci curvature, endowed with the canonical measure. Note however that in the setting of Theorem \ref{theorem-main}, the space is endowed with a generic measure that verifies the (non-smooth) curvature-dimension condition ${\sf RCD}(0,N)$, significantly broadening  the spectrum of applications.     
		 
		Having inequality \eqref{inequality}, the natural questions of sharpness and rigidity arise;  we have the following affirmative answer, which  is new even in the smooth setting:

		\begin{theorem}\label{theorem-main-sharpness-rigidity} {\rm (Sharpness \& Rigidity)}
		Under  the assumptions of Theorem \ref{theorem-main}, the following properties hold. 
		\begin{itemize} 
			\item[(i)] \emph{Sharpness.}   The inequality \eqref{inequality} is sharp  in the sense that for every $N\in (2-\varepsilon_0,2+\varepsilon_0)\cup (3-\varepsilon_0,3+\varepsilon_0)$, $v>0$ and  $\alpha>0$ there exists an ${\sf RCD}(0,N)$ space $(X, {\sf d}, {\sf m})$ with ${\sf AVR}_ {\sf m}=\alpha$ and a domain $\Omega\subset X$ with ${\sf m}(\Omega)=v$ achieving the equality in  \eqref{inequality}.  More precisely,  one can construct such a $(X, {\sf d}, {\sf m})$ as a suitable ${\sf RCD}(0,N)$ metric measure cone, and $\Omega\subset X$ as a metric ball centred at the tip of such a cone. Moreover, if $(X, {\sf d}, {\sf m})$ is an ${\sf RCD}(0,N)$ space where $\rho={\sf d}(x_0,\cdot)$ is subharmonic on $X$ for some $x_0\in \Omega$, i.e., $\Delta \rho \geq 0$, then the constant $ {\sf AVR}_ {\sf m}^{4/N}$ in \eqref{inequality} is sharp.   
			\item[(ii)]   \emph{Rigidity.} 
  If  equality holds in \eqref{inequality} for some open bounded set $\Omega\subset X$, then $(X, {\sf d}, {\sf m})$ is isomorphic as metric measure space to an $N$-Euclidean metric measure cone over an ${\sf RCD}(N-2,N-1)$ space. Moreover, $\Omega$ is isometric to the metric ball $B_{{\sf AVR}_ {\sf m}^{-{1}/{N}}R}(\overline x)$, where ${\sf m}(\Omega)=\omega_N R^N$ and  the point  $\overline x$ is  a tip for the cone structure of $X$; furthermore, up to a multiplicative factor, the eigenfunction realizing equality in \eqref{inequality} has the form 		
 \begin{equation*}
 u(x)=u^*\left({\sf AVR}_ {\sf m}^{\frac{1}{N}}{\sf d}(\overline x,x)\right),\quad \text{for all }  x\in B_{{\sf AVR}_ {\sf m}^{-{1}/{N}}R}(\overline x)\simeq \Omega,
 \end{equation*}
			where 
	\begin{equation}\label{eq:defu*Intro}
 u^*(s)=s^{-\nu}\left(I_\nu(h_\nu) J_\nu\left(h_\nu\frac{s}{R} \right)-J_\nu(h_\nu)  I_\nu\left(h_\nu\frac{s}{R} \right)\right), \ s\in [0,R].
    \end{equation}
		\end{itemize}
		\end{theorem}

	Being in the setting of ${\sf RCD}(0,N)$ spaces, we always have the Laplace comparison  $\Delta \rho \leq \frac{N-1}{\rho}$ on $X\setminus \{x_0\}$, see  Gigli \cite{Gigli1}. 
	In Theorem \ref{theorem-main-sharpness-rigidity}, the additional technical assumption with respect to Theorem \ref{theorem-main}  is the subharmonicity of $\rho$, which is needed in a limiting estimate. Note that we can replace this assumption by a weaker one, requiring only that  $\Delta \rho \geq \frac{1-N}{\rho}$ on $X\setminus \{x_0\}$; however, we prefer the former, more familiar version, noticing as well that our examples verify the  subharmonicity of $\rho$. We also remark that such an extra assumption for proving sharpness in higher order functional inequalities is not so surprising; indeed, the assumption $\Delta \rho \geq \frac{N-5}{\rho}$ on $M\setminus \{x_0\}$, called as the ``distance Laplacian
	growth condition", was crucial in the sharpness of  the classical Sobolev inequality  involving the bi-Laplace operator on $N$-dimensional Riemannian manifolds $(M,g)$ with non-negative Ricci curvature, $N\geq 5$, see Barbosa and Krist\'aly \cite{Barbosa-Kristaly}. Moreover, the subharmonicity of the distance is satisfied on cones: indeed,   on an ${\sf RCD}(0,N)$ metric measure cone with tip at $x_{0}$, it holds that  $\Delta \rho = \frac{N-1}{\rho}\geq 0$.
		
	 Since  the equality case in \eqref{inequality} is characterized, it is also natural to ask about stability; in this respect, we establish the following result:

		\begin{theorem}\label{theorem-main-stability} {\rm (Stability)}
		Let $\varepsilon_{0}>0$ be as in Theorem \ref{theorem-main}, fix $N\in (2-\varepsilon_{0}, 2+ \varepsilon_{0}) \cup (3-\varepsilon_{0}, 3+ \varepsilon_{0})$ and $\alpha,v_0, V>0$. For every $\eta>0$ there exists $\delta>0$ depending on $\eta, N, \alpha, v_0, V$ with the following properties.
  Let  $(X, {\sf d}_{X}, {\sf m}_{X}, \overline x)$ be a pointed  ${\sf RCD}(0,N)$ space with ${\sf AVR}_ {{\sf m}_{X}}=\alpha$ and ${\sf m}(B_{1}(\overline x))\leq v_0$, and let  $B_r(\overline x)$ be a metric ball with ${\sf m}(B_r(\overline x))=V$.
  \begin{enumerate}
\item[(i)] \emph{Stability for the shape of an almost optimal space.}		If 
	\begin{equation}\label{eq:ASSMdelta}
		\Lambda_{\sf m}(B_r(\overline x)) \leq \alpha^\frac{4}{N} h_\nu^4 \left(\frac{\omega_N}{V} \right)^\frac{4}{N} + \delta, 
		\end{equation}
		then there exists an $N$-Euclidean metric measure cone $(Y, {\sf d}_{Y}, {\sf m}_{Y})$   over an ${\sf RCD}(N-2,N-1)$ space, with tip  $y_{0}\in Y$,  such that 
		\begin{equation}\label{eq:Thesiseta}
		 {\sf d}_{\rm pmGH}\left( (X, {\sf d}_{X}, {\sf m}_{X}, \overline x), (Y, {\sf d}_{Y}, {\sf m}_{Y}, y_{0})  \right)\leq \eta,
		\end{equation}
		where  ${\sf d}_{\rm pmGH}$ denotes the pointed measured Gromov--Hausdorff distance between pointed metric measure spaces. Moreover, setting $R= (V / \omega_N)^{1/N}$, it holds that
 \begin{equation}\label{eq:r-Rleqeta}
  |r-\alpha^{-{1}/{N}}R|\leq \eta.
  \end{equation}

\item[(ii)] \emph{Stability for the shape of an almost optimal function.}	  If $u\in H^{2,2}_0(B_r(\overline x))\setminus\{0\}$ is almost optimal, i.e.,
  \begin{equation}
  	\frac{\displaystyle\int_{B_r(\overline x)} (\Delta u)^2 \, {\rm d}{\sf m}}{\displaystyle\int_{B_r(\overline x)} u^2\, {\rm d}{\sf m}} \leq \alpha^\frac{4}{N} h_\nu^4 \left(\frac{\omega_N}{V} \right)^\frac{4}{N} +  \delta, 
  \end{equation}
  then
  \begin{equation}
 \frac{\|u- c\, \bar u\|_{H^{1,2}(X, {\sf d}_X, {\sf m}_X)}} {\|u\|_{L^2(X, {\sf m}_X)}} \leq \eta, 
  \end{equation}
  where $u$ is extended to the value $0$ outside of $B_{r}(\overline x)$, $c\in \mathbb{R}$ is a multiplicative constant depending on $u$, and $\bar u$ is defined by 
  \begin{equation*}
\bar u(x)=u^*\left(\alpha^{\frac{1}{N}}{\sf d}(\overline x,x)\right),\; \text{for all }x\in B_{\alpha^{-{1}/{N}}R}(\overline x), \quad \bar u(x)=0 \quad  \text{otherwise},
 \end{equation*}
 $u^*$ being given in \eqref{eq:defu*Intro} and  $R= (V / \omega_N)^{1/N}$.
 \end{enumerate}
	\end{theorem}

Theorem \ref{theorem-main-stability} is proved by contradiction. Using  Gromov pre-compactness theorem, the stability of the ${\sf RCD}(0,N)$ condition under pmGH convergence, the upper semi-continuity of the asymptotic volume ratio under pmGH convergence (see Lemma \ref{lem:USC-AVR}) as well as a version of Rellich--Kondrachov theorem in varying spaces proved by Ambrosio and Honda \cite{Ambrosio-Honda}, we show that if there exists a sequence $X_j$ contradicting the statement of Theorem \ref{theorem-main-stability} then, up to subsequences, it pmGH-converges to a space $Y$ which satisfies the assumptions of the rigidity Theorem \ref{theorem-main-sharpness-rigidity}/(ii). It follows that $Y$ is a cone, yielding a contradiction.
The proof of the second claim about the shape of an almost optimal function is analogous, after using a result by Nobili and Violo \cite{NV-Adv} about convergence of functions defined on varying spaces.

	\section{Preliminaries}
	\label{section-2}

	\subsection{Sobolev spaces on ${\sf RCD}(0,N)$ spaces}
	Let $(X,{\sf d}, {\sf m})$ be a metric measure space and $\Omega\subseteq X$ be an open  set. For further use, let $\mathsf{LIP}(\Omega)$  (resp.  $\mathsf{LIP_c}(\Omega)$, and $\mathsf{LIP}_{\rm loc}(\Omega)$) be the space of real-valued Lipschitz  (resp.\  compactly  supported Lipschitz,  and  locally Lipschitz) functions over $\Omega$. In addition, we use the standard notations $L^p(X,{\sf m})$ for $L^p$-spaces. Given $u\in \mathsf{LIP}_{\rm loc}(X)$, its slope at $x\in X$ is defined by  $$|Du|(x)=\limsup_{y\to x}\frac{|u(y)-u(x)|}{{\sf d}(x,y)}.$$

	 The $2$-\textit{Cheeger energy} ${\sf Ch}:L^2(X,{\sf m})\to [0,\infty]$ is defined as the convex and lower semicontinuous functional 
	\begin{equation}\label{Cheeger-def}
		{\sf Ch}(u)=\inf\left\{\liminf_{k\to \infty} \int_X |D u_k|^2 {\rm d} {\sf m}:u_k\in {\sf LIP}(X)\cap L^2(X,{\sf m}),\ u_k\to u\ {\rm in}\ L^2(X,{\sf m})  \right\},
	\end{equation}
	see  Cheeger \cite{Cheeger} and Ambrosio,  Gigli and Savar\'e \cite{AGS}.
	Then  $$H^{ 1,2}(X,{\sf d},{\sf m})=\{u\in L^2(X,{\sf m}):{\sf Ch}(u)<\infty\}$$ is  the $L^2$-Sobolev space over $(X,{\sf d}, {\sf m})$, endowed with the norm 
	$\|u\|_{H^{ 1,2}}=\left(\|u\|^2_{L^2(X,{\sf m})}+{\sf Ch}(u)\right)^{1/2}. $
	One can define the minimal $ {\sf m}$-a.e.\ object $|\nabla u | \in  L^2(X, {\sf m})$, the so-called  \textit{minimal $2$-weak upper gradient of} $u \in H^{ 1,2}(X,{\sf d},{\sf m}),$ such that 
	$${\sf Ch}(u)= \int_X |\nabla u|^2\, {\rm d} {\sf m},$$
	and one has \begin{equation}\label{Sobolev-second-definition}
		H^{ 1,2}(X,{\sf d},{\sf m})=\overline{{\sf LIP}_c(X)}^{\|\cdot\|_{H^{ 1,2}}}.
	\end{equation}
	In general,  $H^{ 1,2}(X,{\sf d},{\sf m})$ is a reflexive Banach space. 
	
		Let	$P_2(X,{\sf d})$ be the
	$L^2$-Wasserstein space of probability measures on $X$, while
	$P_2(X,{\sf d}, {\sf m})$ is  the subspace of
	$ {\sf m}$-absolutely continuous measures on $X$.
	For  $N> 1,$ let 
	${\rm Ent}_N(\cdot| {\sf m}):P_2(X,{\sf d})\to \mathbb R$ be  
	the {\it R\'enyi entropy functional} with
	respect to the measure $ {\sf m}$, defined by 
	\begin{equation}\label{entropy}
		{\rm Ent}_N(\nu| {\sf m})=-\int_X \rho^{-\frac{1}{N}}{\rm d}\nu=-\int_X \rho^{1-\frac{1}{N}}{\rm d} {\sf m},
	\end{equation}
	where $\rho$ denotes the density function of $\nu^{\rm ac}$ in
	$\nu=\nu^{\rm ac}+\nu^{\rm s}=\rho  {\sf m}+\nu^{\rm s}$, while $\nu^{\rm ac}$ and $\nu^{\rm s}$
	represent the absolutely continuous and singular parts of $\nu\in
	P_2(X,{\sf d}),$ respectively.
	The \textit{curvature-dimension condition} ${\sf CD}(0,N)$
	states that for all $N'\geq N$ the functional ${\rm Ent}_{N'}(\cdot|\,{\sf m})$ is
	convex on the $L^2$-Wasserstein space $P_2(X, {\sf d}, {\sf m})$, see Lott and Villani \cite{LV} and Sturm \cite{Sturm-2}.

	 If $(X,{\sf d}, {\sf m})$ is a metric measure space satisfying the ${\sf CD}(0,N)$ condition for some $N>1$ and $H^{ 1,2}(X,{\sf d},{\sf m})$
	is a Hilbert space (equivalently, ${\sf Ch}$ from \eqref{Cheeger-def} is a  quadratic form), we say that $(X,{\sf d},{\sf m})$  verifies the \textit{Riemannian curvature-dimension
		condition} ${\sf RCD}(0,N)$, see Ambrosio, Gigli and Savar\'e \cite{AGS2} and Gigli \cite{Gigli1} (see also Ambrosio,  Gigli,  Mondino and Rajala \cite{AGMR}).  In such case, the operation
	$$\nabla u_1 \cdot \nabla u_2=\lim_{\varepsilon\to 0}\frac{|\nabla u_1 +\varepsilon \nabla u_2|^2-|\nabla u_1|^2}{2\varepsilon}$$
	provides a symmetric bilinear form on $H^{ 1,2}(X,{\sf d},{\sf m})\times H^{ 1,2}(X,{\sf d},{\sf m})$.
	
	The \textit{Laplacian operator} $\Delta:L^2(X,{\sf m})\supset \mathcal D(\Delta) \to L^2(X,{\sf m})$ is defined as follows. Let $\mathcal D(\Delta)$ be the set of functions $u\in H^{ 1,2}(X,{\sf d},{\sf m})$ such that for some $h\in L^2(X,{\sf m})$ one has 
	$$\int \nabla u \cdot \nabla g \, {\rm d}{\sf m}=-\int    hg \, {\rm d}{\sf m},\ \ \forall g\in H^{ 1,2}(X,{\sf d},{\sf m});$$ in this case, we set $h=\Delta u.$ In the setting of  ${\sf RCD}(0,N)$ spaces (thus ${\sf Ch}$ is a  quadratic form), it turns out that $u\mapsto \Delta u$ is linear. 
	
 The \textit{perimeter} of the Borel set $\Omega\subset X$ is defined as 
	$${\rm Per}(\Omega)=\inf\left\{\liminf_{k\to \infty} \int_X |D u_k| {\rm d} {\sf m}:u_k\in \mathsf{LIP}_{\rm loc}(X),\ u_k\to\chi_\Omega\ {\rm in}\ L^1_{\rm loc}(X) \right\}.$$
 
	Given an ${\sf RCD}(0,N)$ space $(X,{\sf d},{\sf m})$ and $\Omega\subset X$ an open domain, for every Borel function $f:\Omega\to \mathbb R$ and nonnegative $v\in {\sf LIP}(\Omega)$, the following form of the \textit{co-area formula} holds: 
	\begin{equation}\label{co-area-formlula}
		\int_{\{s\leq v<r\}}f|Dv|\,{\rm d}{\sf m}= \int_s^r\left(\int f {\rm d}\,{\rm Per}(\{v>t\})\right){\rm d}t,\ \ \forall s\in [0,r],
	\end{equation} see Miranda \cite[Proposition 4.2 \& Remark 4.3]{Miranda} (see also Mondino and Semola \cite[Corollary 2.13]{MSemola}). By Gigli \cite[Theorem 5.3]{Gigli1}, if $x_0\in X$ is fixed and $\rho(x)={\sf d}(x_0,x)$,  we have the eikonal equation
\begin{equation}\label{eikonal}
	|\nabla \rho|=1\ \ {\sf m}{\rm -a.e.}
\end{equation}
Let $B_r(x)=\{y\in X: {\sf d}(x,y)<r\}$ be  the open metric ball in $X$  with center $x\in X$ and radius $r>0.$  If $f:[0,r]\to \mathbb R$ is an integrable function and $v:=\rho={\sf d}(x_0,\cdot)$, by using \eqref{eikonal}, the relation  \eqref{co-area-formlula} reduces to
 	\begin{equation}\label{area-formlula}
 	\int_{B_r(x_0)}f\circ \rho\,{\rm d}{\sf m}= \int_0^r f(t)\, {\sf m}^+(B_t(x_0))\,  {\rm d}t,\ \ \forall r\geq 0,
 \end{equation}
where  ${\sf m}^+$ is the Minkowski content, i.e., for every open bounded $\Omega\subset X$, 
$${\sf m}^+(\Omega)=\liminf_{\epsilon\to 0}\frac{{\sf m}(\Omega_\epsilon\setminus\Omega)}{\epsilon},$$
the set  $\Omega_\epsilon=\{x\in X:\inf_{y\in \Omega}{\sf d}(x,y)<\epsilon\}$ being the $\epsilon$-neighborhood of $\Omega$, $\epsilon>0.$  Indeed, by combining Lebesgue theorem,  the co-area formula \eqref{co-area-formlula} and the eikonal equation \eqref{eikonal}, it follows that
\begin{equation}\label{eq:m+=PerBall}
{\sf m}^+(B_r(x_0))=\liminf_{\epsilon\to 0}\frac{1}{\epsilon} \int_r^{r+\varepsilon} {\rm d}\,{\rm Per}(\{\rho>t\})\, {\rm d}t= {\rm Per}(B_r(x_0)), \quad \text{for $\mathcal L^1$-a.e.}\ r>0.
\end{equation}

	According to Bruè,  Pasqualetto and Semola \cite{BPS}, the \textit{Gauss--Green formula} on an  ${\sf RCD}(0,N)$ space $(X,{\sf d},{\sf m})$ reads as follows: if $E\subset X$ is a bounded set of finite perimeter, then there exists a
	unique vector field $\nu_E\in L^2_E(TX)$ with $|\nu_E|=1$ Per$(E)$-a.e.\ such that 
	\begin{equation}\label{Gauss-Green}
		\int_E {\rm div}(v){\rm d}{\sf m}=-\int \langle {\rm tr}_E(v), \nu_E \rangle {\rm d}{\rm Per}(E),
	\end{equation}
	for every $v\in L^2(TX)\cap D({\rm div})$ with $|v|\in L^\infty(X,{\sf m}),$ where ${\rm tr}_E$ stands for the
	trace operator over the boundary of $E$.

	Let $(X,{\sf d}, {\sf m})$ be a metric measure space verifying  the ${\sf RCD}(0,N)$ condition and $\Omega\subset X$ be an open bounded set. Denote by ${\rm Dom}(\Delta)= \{ u\in L^2(X,{\sf m}) \,: \, \Delta u \in L^2(X, {\sf m}) \}$ the finiteness domain of the Laplacian, i.e., the space of $L^2$ functions whose distributional Laplacian belongs to $L^2(X, {\sf m})$.  Consider the space of test functions
	$$\mathsf{Test}_c(\Omega):=\{ u\in {\sf LIP}(\Omega) : u\in {\rm Dom}(\Delta), \; {\rm supp\,} u \subset \Omega \}.$$
	According to Ambrosio,  Mondino and Savaré \cite[Lemma 6.7]{AMS}, the space $\mathsf{Test}_c(\Omega)$ contains a large class of nonzero functions. 
	By combining the Poincar\'e inequality (see Rajala \cite{Rajala}) with an integration by parts and H\"older inequality, one can find a constant $C>0$, depending only on $N$ and ${\sf m}(\Omega)$,   such that for  every   $u\in  \mathsf{Test}_c(\Omega)$,  
	\begin{equation}\label{eq:unablauL2}
		\int_\Omega  u^2 {\rm d}{\sf m}+\int_\Omega  |\nabla u|^2 {\rm d}{\sf m}
		\leq C \int_\Omega  (\Delta u)^2 {\rm d}{\sf m}. 
	\end{equation}
	Moreover, every $u\in  \mathsf{Test}_c(\Omega)$ admits an $L^2$-Hessian and since  $(X,{\sf d}, {\sf m})$ verifies the ${\sf RCD}(0,N)$ condition, a Bochner-type estimate implies that 
	\begin{equation}\label{eq:HessuL1}
		\int_\Omega |\textrm{Hess}\, u|_{\rm HS}^2 {\rm d}{\sf m} 	\leq \int_\Omega  (\Delta u)^2 {\rm d}{\sf m},
	\end{equation}
	where $|\cdot|_{\rm HS}$ stands for the Hilbert--Schmidt pointwise norm,  see Gigli \cite[Corollary 2.10]{Gigli2}. 	Combining \eqref{eq:unablauL2} and \eqref{eq:HessuL1} we  obtain  for every $u\in  \mathsf{Test}_c(\Omega)$ that
	\begin{equation}\label{eq:HessuL2}
		\|u\|_{W^{2,2}(\Omega,{\sf m})}^2:=	\int_\Omega  u^2 {\rm d}{\sf m}+\int_\Omega  |\nabla u|^2 {\rm d}{\sf m}+\int_\Omega |\textrm{Hess}\, u|^2_{\rm HS} {\rm d}{\sf m} 	\leq (C+1) \int_\Omega  (\Delta u)^2 {\rm d}{\sf m}.
	\end{equation}
	The latter estimate enables us to define the space of functions $H^{2,2}_0(\Omega,{\sf m})$ to be the closure of $\mathsf{Test}_c(\Omega)$ with respect to	the Hilbertian norm
	$$
	\|u\|_{H^{2,2}_0}:= \sqrt{\int_{\Omega} (\Delta u)^2 {\rm d}{\sf m}}.
	$$
	By this argument,  $H^{2,2}_0(\Omega,{\sf m})$ is a natural non-smooth counterpart of the usual Sobolev space $W_0^{2,2}(\Omega)$ from the smooth setting of Euclidean/Riemannian manifolds, see Hebey \cite[Proposition 3.3]{Hebey}. The following observation is crucial in the study of clamped plates on ${\sf RCD}(0,N)$ spaces. 
	
	\begin{proposition}\label{prop-laplace=00} Let $(X,{\sf d},{\sf m})$ be an    ${\sf RCD}(0,N)$ space and $\Omega\subset X$ be an open bounded domain. Then  
		\begin{equation}\label{eq:intDeltau=0}
			\int_\Omega \Delta u\, {\rm d}{\sf m}=0,\ \ \ \text{for \ all}\ u\in H^{2,2}_0(\Omega, {\sf m}).
		\end{equation}
	\end{proposition}
	\begin{proof}
		We first claim that
		\begin{equation}\label{eq:intDeltau0Test}
			\int_\Omega \Delta u\, {\rm d}{\sf m}=0, \quad \text{for all } u\in \mathsf{Test}_c(\Omega). 
		\end{equation}
		In order to prove \eqref{eq:intDeltau0Test}, let us first observe that there exists an open set of finite perimeter $\hat{\Omega}$ such that ${\rm supp\,} u \subset \hat{\Omega} \subset \Omega$. Indeed, from the fact that ${\rm supp\,} u$ is compactly contained in $\Omega$ (which is a consequence of the boundedness of $\Omega$ and the fact that $(X,{\sf d})$ is proper), then it is at positive distance from $X\setminus \Omega$. The distance function from $X\setminus \Omega$ is 1-Lipschitz and thus, by the co-area formula  \eqref{co-area-formlula}, it follows that ${\mathcal L}^1$-a.e.\;superlevel set has finite perimeter. One can then choose $\hat{\Omega}$ to be one of such superlevel sets of the distance from $X\setminus \Omega$ with finite perimeter.
		
		By the Gauss--Green formula \eqref{Gauss-Green}, we have that
		\begin{equation*}
			\int_\Omega \Delta u\, {\rm d}{\sf m}= \int_{\hat{\Omega}}  \Delta u\, {\rm d}{\sf m}= -\int \langle {\rm tr}_{\hat{\Omega}}(\nabla u), \nu_{\hat{\Omega}} \rangle {\rm d}{\rm Per}(\hat{\Omega})=0, 
		\end{equation*}
		where the last identity follows by the fact that $u$ vanishes in a neighbourhood of $\partial \hat{\Omega}$, and thus  $\langle {\rm tr}_{\hat{\Omega}}(\nabla u), \nu_{\hat{\Omega}}\rangle=0$ ${\rm Per}(\hat{\Omega})$-a.e.\  This completes the proof of \eqref{eq:intDeltau0Test}.
		
		Now let us prove \eqref{eq:intDeltau=0}.
		Fix $u\in H^{2,2}_0(\Omega, {\sf m})$. By the very definition of $H^{2,2}_0(\Omega, {\sf m})$, there exists a sequence $(u_k)_k\subset \mathsf{Test}_c(\Omega)$ such that $\Delta u_k\to \Delta u$ in $L^2(\Omega, {\sf m})$. Recalling that $\Omega$ is bounded and thus it has finite measure, we infer that $\Delta u_k\to \Delta u$ in $L^1(\Omega, {\sf m})$ and thus
		$$
		\int_\Omega \Delta u\, {\rm d}{\sf m}=\lim_{k\to \infty} \int_\Omega \Delta u_k\, {\rm d}{\sf m}=0,
		$$
		where, in the last identity, we used \eqref{eq:intDeltau0Test}. 
	\end{proof}

\subsection{Rearrangements.} Let $(X,{\sf d}, {\sf m})$ be a metric measure space verifying the ${\sf RCD}(0,N)$ condition and  $\Omega\subset X$ be an open bounded domain. For a Borel measurable function $u\colon 
	\Omega\to[0,\infty)$ we consider its \textit{distribution function} $\mu:[0,\infty)\to [0,{\sf m}(\Omega)]$ given by 
	$$\mu(t)={\sf m}(\{u>t\}).$$
	Note that $\mu$ is non-increasing and left-continuous. Let $u^\#$ be the \textit{generalized inverse of} $\mu$ given by
	$$u^\#(s)=\begin{cases}
		{\rm ess\,sup}\, u, &\text{ if }s=0,\\
		\inf\{t:\mu(t)<s\}, &\text{ if }s>0.
	\end{cases}$$
	For every Borel  function $u\colon \Omega\to[0,\infty)$ the \textit{monotone rearrangement} ${u^*}\colon[0,r]\to[0,\infty)$ is defined as
	$$u^*(s)=u^\#(\omega_N s^N),\ \ s\in [0,r],$$
	where  $r>0$ is chosen so that ${\sf m}(\Omega)=\omega_N r^N.$ If we consider the 1-dimensional model space $([0,\infty),{\sf d}_{\rm eu},\sigma_N=N\omega_Nr^{N-1}\mathcal L^1)$, one has  that
	 $$\mu(t)=\sigma_N(\{u^*>t\})=\mathcal L^1(\{u^\#>t\}), \quad \text{for every $t>0$}, $$
	where ${\sf d}_{\rm eu}=|\cdot|$ is the usual Euclidean distance on $[0,\infty)$. 
		The layer cake representation gives that 
		\begin{equation}\label{equimeasure}
		\|u\|_{L^p(\Omega,{\sf m})}=\|u^*\|_{L^p(\Omega^*,\sigma_N)},  \quad \text{for every } p\geq 1,
	\end{equation} 
	where $\Omega^*=[0,r]$, with $ {\sf m}(\Omega)=\sigma_N([0,r])=\omega_N r^N.$ 
	If $S\subset \Omega$ is open, then a Hardy--Littlewood-type estimate shows that  
	\begin{equation}\label{Hardy-Littlewood}
		\|u\|_{L^1(S,{\sf m})}\leq \|u^*\|_{L^1(S^*,\sigma_N)},
	\end{equation}
	with equality when $S= \Omega$.

	\subsection{Sharp isoperimetric inequality on ${\sf RCD}(0,N)$ spaces} \label{subsection-2-3}	If $(X,{\sf d}, {\sf m})$ is a metric measure space satisfying the ${\sf CD}(0,N)$ condition for some $N>1$, then the Bishop--Gromov comparison principle states that the functions  
	\begin{equation}\label{Bishop-Gromov-monoton}
		r\mapsto \frac{ {\sf m}(B_r(x))}{r^{N}},\ \ \ r\mapsto \frac{ {\sf m}^+(B_r(x))}{r^{N-1}},\ \ r>0,
	\end{equation}
	are non-increasing on $[0,\infty)$ for every $x\in X$, see Sturm \cite{Sturm-2}.
	An important consequence is that the \textit{asymptotic volume ratio}
	\begin{equation}\label{AVR-difinition}
		{\sf AVR}_ {\sf m}=\lim_{r\to \infty}\frac{ {\sf m}(B_r(x))}{\omega_Nr^N}=\lim_{r\to \infty}\frac{ {\sf m}^+(B_r(x))}{N\omega_Nr^{N-1}},
	\end{equation}
	is well-defined, i.e., it is independent of the choice of $x\in X$. Here,  $\omega_N=\frac{\pi^{\frac{N}{2}}}{\Gamma(\frac{N}{2}+1)}$ is a scaling factor, which coincides with the volume of the $N$-dimensional unit ball in $\mathbb R^N$, whenever $N\in \mathbb N.$

	If ${\sf AVR}_ {\sf m}>0 $, we have the \textit{sharp isoperimetric inequality} on $(X,{\sf d}, {\sf m})$ verifying the curvature-dimension condition ${\sf CD}(0,N)$, cf.\ Balogh and Krist\'aly \cite{BK}; namely,  for  every bounded Borel subset $\Omega\subset X$ it holds  
	\begin{equation}\label{eqn-isoperimetric-2}
		{\rm Per}(\Omega)\geq N\omega_N^\frac{1}{N}{\sf AVR}_ {\sf m}^\frac{1}{N} {\sf m}(\Omega)^\frac{N-1}{N},
	\end{equation}
	and the constant $N\omega_N^\frac{1}{N}{\sf AVR}_ {\sf m}^\frac{1}{N}$ in  \eqref{eqn-isoperimetric-2} is sharp. 
	We notice that the initial form of \eqref{eqn-isoperimetric-2} from \cite{BK} has been proved for the Minkowski content ${\sf m}^+$ instead of the perimeter ${\rm Per}$; in fact,  a  suitable approximation argument also shows the present form, see e.g.\ Ambrosio, Di Marino and Gigli \cite[Theorem 3.6]{ADMG} or Nobili and Violo \cite[Proposition 3.10]{NV}.  
	
	We notice that in the setting of ${\sf RCD}(0,N)$ spaces, the equality case in \eqref{eqn-isoperimetric-2} has been recently characterized, stating that equality holds in \eqref{eqn-isoperimetric-2} for some $\Omega\subset X$ if and only if the following two properties simultaneously  hold: 
	\begin{itemize}
		\item[(i)] $X$ is isometric to an $N$-Euclidean   metric measure cone over an ${\sf RCD}(N-2,N-1)$ space, i.e., there exists a compact ${\sf RCD}(N-2,N-1)$ metric measure space $(Z,{\sf d}_Z,{\sf m}_Z)$   such that  $(X,{\sf d}, {\sf m})$ is isometric to the metric measure cone $(C(Z),{\sf d}_c,t^{N-1}{\rm d}t\otimes {\sf m}_Z)$, where $$C(Z)=Z \times [0,\infty)/(Z \times \{0\})$$ and ${\sf d}_c$ is the usual cone distance from ${\sf d}_Z$.  Recall that an ${\sf RCD}(N-2,N-1)$ space is an infinitesimally Hilbertian metric measure space satisfying the ${\sf CD}(N-2, N-1)$ condition, i.e., having Ricci curvature bounded below by $N-2$ and dimension bounded above by $N-1$, in a synthetic sense. Indeed it was proved by Ketterer \cite{Ketterer} that if an ${\sf RCD}(0,N)$ space $X$ is an $N$-cone on a metric measure space $Y$, then $Y$ is an ${\sf RCD}^*(N-2,N-1)$ space; afterwards, it was proved by Cavalletti and Milman \cite{C-M-Inv} that any ${\sf RCD}^*(N-2,N-1)$ space satisfies the (a priori stronger)  ${\sf RCD}(N-2,N-1)$ condition.   
		\item[(ii)] $\Omega$ is (up to an ${\sf m}$-negligible set) a metric ball centered
		at one of the tips of $X;$ here, the point  $Z \times \{0\}$ is called the tip of the cone.
	\end{itemize}   This characterization has been stated by Antonelli,  Pasqualetto,  Pozzetta and  Violo \cite{APPV}; for previous forms with slightly different assumptions, see Antonelli,  Pasqualetto,   Pozzetta and Semola \cite{APPS} and Cavalletti and  Manini \cite{C-M}.

	\section{Coupled minimization problem: from ${\sf RCD}(0,N)$ to the model space}
	
	\subsection{Fine level-set analysis} In this subsection we establish some important, preparatory results for the level sets of the eigenfunction associated to the variational problem \eqref{Rayleigh}. 
	
	\begin{proposition}\label{prop-inf-min} The infimum in \eqref{Rayleigh} is achieved. 
	\end{proposition}
	
	\begin{proof}
				By \eqref{eq:unablauL2}, we know that  $H^{2,2}_0(\Omega,{\sf m})$ continuously embeds  into $H^{ 1,2}(X,{\sf d},{\sf m})$ which in turn, compactly embeds into $L^2(X,{\sf m})$, see Gigli,  Mondino and Savaré \cite{GMS}.  
				The existence of a minimizer for \eqref{Rayleigh} follows then by direct methods in calculus of variations, weak compactness and lower semicontintuity of the norm in the Hilbert space $H^{2,2}_0(\Omega,{\sf m})$ to minimize $u\mapsto \|u\|_{H^{2,2}_0}$ under the constraint $\|u\|_{L^2(\Omega,{\sf m})}=1$.
	\end{proof}
	From now on, we denote by $u\in H^{2,2}_0(\Omega,{\sf m})$  the eigenfunction achieving the infimum in \eqref{Rayleigh}.
	
	\begin{proposition}\label{prop-lipschitz}
		If $N\in [1,4)$, then $u$ has a locally Lipschitz representative in $\Omega$. 
	\end{proposition}
	\begin{proof}
		Notice that $u\in H^{2,2}_0(\Omega,{\sf m})$ solves $\Delta^2 u=\Lambda_{\sf m}(\Omega) u$. Calling $v:=\Delta u$, we have that $\Delta v= \Lambda_{\sf m}(\Omega) u\in L^2(\Omega)$ and  we infer that $v\in L^\infty(\Omega)$; note that for this argument, the condition $\frac{N}{2}<2$ was crucial, see Mondino and Vedovato \cite[Theorem 5.1]{MV-CVPDE}. The thesis follows from  the locally Lipschitz  regularity of the solutions to the Poisson equation, see Jiang \cite[Theorem 1.2]{Jiang}, combined with the fact that $\Delta u=v\in L^\infty(\Omega)$.
	\end{proof}
	
	In the rest of the paper, unless otherwise specified, we will always work with the locally Lipschitz representative of $u$. This is well-motivated, as in Theorem \ref{theorem-main} we assume that $N<4$. Since $u$ can have nodal domains, i.e., it could be sign-changing, we consider 
	$$\Omega_+:=\{u>0\}=\{x\in \Omega: u(x)>0\}\ \ {\rm  and}\ \  \Omega_-:=\{u<0\}=\{x\in \Omega: u(x)<0\};$$
	similarly, let 
	$u_+=\max(u,0)$ and $u_-=-\min(u,0)$ be the positive and negative parts of $u$, respectively. Let $\tilde \Omega=\Omega_+\cup \Omega_-$ and $\tilde R>0$ be such that $\omega_N \tilde R^N={\sf m}(\tilde \Omega)$. Finally, let $a,b\geq 0$ be such that 
	\begin{equation}\label{a-b}
		\omega_N a^N={\sf m}(\Omega_+) \ \ {\rm and} \ \  \omega_N b^N={\sf m}(\Omega_-);
	\end{equation}
	thus  $\tilde R^N=a^N+b^N$. It is clear that $\tilde R\leq R$, where $\omega_N R^N={\sf m}(\Omega)$ since we have $\omega_N (R^N-\tilde R^N)={\sf m}(\{u=0\})\geq 0.$  
	From \eqref{Rayleigh}, it follows that
		\begin{equation}\label{Lieb-modified}
		\Lambda_{\sf m}(\Omega)= \frac{\displaystyle\int_\Omega (\Delta u)^2 {\rm d}{\sf m}}{\displaystyle\int_\Omega u^2 {\rm d}{\sf m}}\geq \frac{\displaystyle\int_{\tilde \Omega} (\Delta u)^2 {\rm d}{\sf m}}{\displaystyle\int_{\tilde \Omega} u^2 {\rm d}{\sf m}}.
	\end{equation}
	Therefore, it is enough to restrict the function $u$ to  $\tilde \Omega$; for simplicity, we still keep the same notation $u$ for this restriction. The distribution functions of $u_\pm$ are given by $\mu_\pm(t)={\sf m}(\{u_\pm>t\})$ together with their generalized inverses $u_\pm^\#$ and monotone rearrangements $u_\pm^*;$ in a similar way, we consider $(\Delta u)^\#_\pm$ and $(\Delta u)^*_\pm$, respectively.  For further use, for $t>0$, let $r_t,\rho_t>0$ be such that $\omega_N r_t^N=\mu_+(t)$ and $\omega_N \rho_t^N=\mu_-(t)$, respectively. It is clear from \eqref{a-b} that $r_0:=\lim_{t\to 0}r_t=a$ and  $\rho_0:=\lim_{t\to 0}\rho_t=b.$
	
	\begin{proposition}\label{prop-div-them}
		For $\mathcal L^1$-a.e.\;$t>0$,  it holds that
		\begin{equation}\label{eq:GaussGreen-u}
			\int_{\{u_\pm=t\}}|\nabla u|\, {\rm d}{\rm Per}(\{u_\pm>t\})=-\int_{\{u_\pm>t\}}\Delta u\, {\rm d}{\sf m}.
		\end{equation}
	\end{proposition}
	
	\begin{proof} 
		The proof follows closely Bru\'e, Pasqualetto and Semola \cite[Proposition 6.1]{BPS-ASNS}, where a similar result was proved for distance functions. We recall the argument for the reader's convenience.
		We prove the claim for $'+'$, the case $'-'$  being similar.
  
		By the Gauss--Green formula \eqref{Gauss-Green}, it holds that
		\begin{equation}\label{eq:GGDeltau}
			-\int_{\{u_+>t\}}\Delta u\, {\rm d}{\sf m}=\int_{\{u=t\}} \langle {\rm tr}_{\{u>t\}}(\nabla u), \nu_{\{u>t\}} \rangle {\rm d}{\rm Per}(\{u>t\}).
		\end{equation}
		We will show that the Radon--Nikodym derivative of $\langle {\rm tr}_{\{u>t\}}(\nabla u), \nu_{\{u>t\}} \rangle {\rm d}{\rm Per}(\{u>t\})$ with respect to ${\rm Per}(\{u>t\})$ coincides with $|\nabla u|$  for $\mathcal L^1$-a.e.\;$t>0$. The claim \eqref{eq:GaussGreen-u} will then follow by Lebesgue theorem, recalling that the perimeter measure is asymptotically doubling, see Ambrosio \cite[Corollary 5.8]{Ambrosio2002}. Therefore, we are reduced to show that, for $\mathcal L^1$-a.e.\;$t>0$, it holds that
		\begin{equation}\label{eq:blowupTrace}
			\lim_{r\to 0^+} \frac{\langle {\rm tr}_{\{u>t\}}(\nabla u), \nu_{\{u>t\}} \rangle {\rm Per}_{\{u>t\}}(B_r(x))} {{\rm Per}(B_r(x))} =  |\nabla u| (x), \quad {\rm Per}(\{u>t\})\text{-a.e.}\;x.
		\end{equation}
		
		We claim that \eqref{eq:blowupTrace} follows by a blow-up argument, thanks to the next four facts that will be proved below; namely, for $\mathcal L^1$-a.e.\;$t>0$ and ${\rm Per}(\{u>t\})\text{-a.e.}\;x\in \{u=t\}$ the following statements hold:
		\begin{enumerate}
			\item $x$ is a regular point for $(X,{\sf d}, {\sf m})$, i.e., $(X,{\sf d}, {\sf m})$ has a unique tangent cone at $x$, and such a tangent cone is isomorphic to the Euclidean space $\mathbb R^n$ for some $1\leq n\leq N$;
			\item $x$ is a regular reduced-boundary point for the finite perimeter set $\{u>t\}$, i.e., any blow-up of $\{u>t\}$ at $x$ in the sense of finite perimeter sets is a half-space in $\mathbb R^n$, see Ambrosio, Bru\'e and Semola \cite[\S3]{ABS-GAFA};
			\item $x$ is a regular point for $u$, i.e., any blow-up of the function $u$ at $x$ is a linear function $u_x:{\mathbb R}^n\to {\mathbb R}$ with slope $|\nabla u|(x) $;
			\item denoting by ${\mathbb H}^n\subset {\mathbb R}^n$ the half-space arising as the blow-up of $\{u>t\}$ at $x$ in item (2),  $u_x$ coincides with the signed distance function from ${\mathbb H}^n$, scaled by the real number $|\nabla u|(x)$, i.e., $u_x= |\nabla u| (x) \, d^{\pm}_{{\mathbb H}^n}$.
		\end{enumerate}
		Before discussing the validity of statements  (1)--(4), let us show that they are sufficient to obtain \eqref{eq:blowupTrace} and thus the thesis. By the Gauss--Green formula \eqref{Gauss-Green} we know that $-\langle {\rm tr}_{\{u>t\}}(\nabla u), \nu_{\{u>t\}} \rangle$ coincides with the divergence of the vector field $\chi_{\{u>t\}} \, \nabla u$, which in turn equals $\Delta (u-t)_+$. Assuming that (1)--(4) hold,  the blow-up of $\{u>t\}$ converges to the half-space ${\mathbb H}^n$, and the function $u$ converges to $|\nabla u|(x) \, d^{\pm}_{{\mathbb H}^n}$. In such a limit, it is immediate to verify that
		$$
		\langle {\rm tr}_{{\mathbb H}^n}(|\nabla u| (x) \, \nabla d^{\pm}_{{\mathbb H}^n}), \nu_{{\mathbb H}^n} \rangle = |\nabla u| (x).
		$$
		Scaling and stability of the distributional Laplacian then yields \eqref{eq:blowupTrace}.
		
		It remains to prove that (1)--(4) hold for $\mathcal L^1$-a.e.\;$t>0$ and ${\rm Per}(\{u>t\})\text{-a.e.}\;x\in \{u=t\}$. By the co-area formula \eqref{co-area-formlula}, it suffices to prove that they are satisfied ${\sf m}$-a.e.
		
		Property (1) was proved by Mondino and Naber \cite[Corollary 1.2]{MN-JEMS}. Property (2) was established by Bru\'e, Pasqualetto and Semola \cite[Theorem 3.2]{BPS}. Property (3) follows by  Cheeger \cite[Theorem 10.2]{Cheeger}. We are left to show the validity of (4). To this aim, observe that for all $x\in \{u=t\}$ one has
		\begin{equation}\label{eq:u-int}
			(u-t)_+(x)\leq 0 \quad \text{and} \quad \int_{\{u<t\}\cap B_r(x)}(u-t)_+ \,{\rm d}{\sf m}=0.
		\end{equation}
		The inequalities in \eqref{eq:u-int} are stable under blow-up of the set $\{u>t\}$ and of the function $u$, due to the $L^1_{\rm loc}$-convergence of the indicator functions of the scaled sets. Therefore,  $(u_x)_+$ also vanishes identically on the complement of the blow-up of $\{u>t\}$ at $x$. Since the blow-up of $\{u>t\}$ is the half-space ${\mathbb H}^n\subset {\mathbb R}^n$, and the blow-up of $u$ at $x$ is a linear function with slope $|\nabla u|(x)$, the only possibility is that (4) holds.
	\end{proof}
	
	\begin{proposition}\label{isoperi-prop}
		For $\mathcal L^1$-a.e.\ $t>0$,  it holds 
		$$
		{\rm Per}^2(\{u_\pm>t\})\leq \mu_\pm'(t)\int_{\{u_\pm>t\}}\Delta u\, {\rm d}{\sf m}.
		$$
	\end{proposition}
	\begin{proof} The claim is known in the Euclidean setting, see  Talenti \cite[Appendix, p.\ 278]{Talenti}. In the sequel, we provide the proof in the non-smooth setting for the $'+'$; the case $'-'$  is analogous.    For any $h>0$, Cauchy--Schwarz inequality implies
	$$\left(\frac{1}{h}\int_{\{t<u\leq t+h\}}|\nabla u|{\rm d}{\sf m}\right)^2\leq \frac{\mu_+(t)-\mu_+(t+h)}{h}\frac{1}{h}\int_{\{t<u\leq t+h\}}|\nabla u|^2{\rm d}{\sf m}.$$
	When $h\to 0$, the latter relation and the co-area formula \eqref{co-area-formlula} imply that
	\begin{equation}\label{eq:Per2leq}
		{\rm Per}^2(\{u_\pm>t\}) \leq -\mu_+'(t)\int_{\{u=t\}}|\nabla u|{\rm d}{\rm Per}(\{u_+>t\}).
	\end{equation}
	We conclude combining the last inequality with Proposition \ref{prop-div-them}.
	\end{proof}

\subsection{Nodal decomposition and the appearance of ${\sf AVR}_{\sf m}$} In this subsection we prove nodal decomposition estimates both for the terms $\ds\int_{\tilde \Omega}u^2 {\rm d}{\sf m}$ and $\ds\int_{\tilde \Omega}(\Delta u)^2 {\rm d}{\sf m}$,   where the asymptotic volume ratio ${\sf AVR}_{\sf m}$ plays a crucial role.  Let 
\begin{equation}\label{definition-F-pm}
	F_+(s)=(\Delta u)^\#_-(s)-(\Delta u)^\#_+({\sf m}(\tilde \Omega)-s)\ \ {\rm and}\ \ F_-(s)=-F_+({\sf m}(\tilde \Omega)-s),\ \ s\in [0,{\sf m}(\tilde \Omega)].
\end{equation}
By definition, it follows that at least one of the terms   $(\Delta u)^\#_-(s)$ and $(\Delta u)^\#_+({\sf m}(\tilde \Omega)-s)$ is zero, i.e., 
\begin{equation}\label{simult=0}
	(\Delta u)^\#_-(s)(\Delta u)^\#_+({\sf m}(\tilde \Omega)-s)=0,\ \ s\in [0,{\sf m}(\tilde \Omega)].
\end{equation}

\begin{proposition} \label{f-hasonlitas}
	For every $t>0$ one has that 
	$$\ds\displaystyle\int_0^{\mu_\pm(t)} F_\pm(s)\, {\rm d}s\geq -\int_{\{u_\pm>t \}}\Delta u\,{\rm d}{\sf m}.$$
\end{proposition}

\begin{proof} Again, we consider the case $'+'$, the other one being analogous. 
Let $t>0$ be fixed. It is enough to prove that 
\begin{equation}\label{kell-becsles-1}
	\displaystyle\int_0^{\mu_+(t)} (\Delta  u)^\#_-(s)\, {\rm d}s\geq \int_{\{u_+>t \}}(\Delta  u)_-{\rm d}{\sf m},
\end{equation} 
and 
\begin{equation}\label{kell-becsles-2}
	\displaystyle\int_{\{u_+>t \}}(\Delta u)_+ {\rm d}{\sf m}\geq \int_0^{\mu_+(t)} (\Delta u)^\#_+({\sf m}(\tilde \Omega)-s)\, {\rm d}s.
\end{equation}
We recall that $r_t>0$ is such that $\omega_N r_t^N=\mu_+(t)$. Thus, inequality  (\ref{kell-becsles-1}) follows by a change of variable and \eqref{Hardy-Littlewood} as
\begin{eqnarray*}
	\displaystyle\int_0^{\mu_+(t)} (\Delta u)^\#_-(s){\rm d}s&=&\int_0^{ r_t} (\Delta u)^\#_-(\omega_N z^N){\rm d}\sigma_N (z)=\int_0^{r_t}(\Delta u)^*_-(z){\rm d} \sigma_N (z)\\&\geq & \int_{\{u_+>t \}}(\Delta u)_-{\rm d}{\sf m},
\end{eqnarray*}
where we used that   $\{u_+>t \}^*=[0,r_t]$. The proof of (\ref{kell-becsles-2}) is similar. Indeed, if $\tau_t>0$ is the unique  number verifying $\omega_N\tau_t^N={\sf m}(\tilde \Omega)-\mu_+(t)$, and $S_t:=\{u\leq t\}=\{x\in \tilde \Omega: u(x)\leq t\}$,  then ${\sf m}(S_t)={\sf m}(\tilde \Omega)-\mu_+(t)=\omega_N\tau_t^N$, and a change of variable and  \eqref{Hardy-Littlewood} imply  that 
\begin{eqnarray*}
	\int_0^{\mu_+(t)} (\Delta u)^\#_+({\sf m}(\tilde \Omega)-s){\rm d}s&=& \int_{\tau_t}^{\tilde R}(\Delta u)^*_+(z){\rm d} \sigma_N (z)=\int_{0}^{\tilde R}(\Delta u)^*_+(z){\rm d} \sigma_N (z)-\int_{0}^{\tau_t}(\Delta u)^*_+(z){\rm d} \sigma_N (z)\\&\leq  & \int_{\tilde \Omega}(\Delta u)_+{\rm d}{\sf m}-\int_{S_t}(\Delta u)_+{\rm d}{\sf m}=\int_{\tilde \Omega\setminus S_t}(\Delta u)_+{\rm d}{\sf m}\\&=&\int_{\{u_+>t \}}(\Delta u)_+ {\rm d}{\sf m}.
\end{eqnarray*}
By (\ref{kell-becsles-1}) and (\ref{kell-becsles-2}) one has 
\begin{eqnarray*}
	\int_0^{\mu_+(t)} F_+(s)\, {\rm d}s&=&\int_0^{\mu_+(t)} (\Delta u)^\#_-(s)\, {\rm d}s-\int_0^{\mu_+(t)}(\Delta u)^\#_+({\sf m}(\tilde \Omega)-s)\, {\rm d}s\\&\geq& 
	\int_{\{u_+>t \}}(\Delta  u)_- \, {\rm d}{\sf m}-\int_{\{u_+>t \}}(\Delta u)_+\,  {\rm d}{\sf m}\\
	&=&-\int_{\{u_+>t \}}\Delta u\, {\rm d}{\sf m},
\end{eqnarray*}
which is precisely the required claim. 
\end{proof}

Let $a$ and $b$  from {\rm (\ref{a-b})}. We consider the functions $V_+:[0,a]\to \mathbb R$ and $V_-:[0,b]\to \mathbb R$ defined by 
\begin{equation}\label{v-definit}
	V_+(x)=\frac{1}{N\omega_N}\int_{x}^a \rho^{1-N}\left(\int_0^{\omega_N \rho^N}F_+(s){\rm d}s\right){\rm d}\rho,\ \ x\in [0,a],
	\end{equation}
	and
	\begin{equation}\label{w-definit}
		V_-(x)=\frac{1}{N\omega_N}\int_{x}^b \rho^{1-N}\left(\int_0^{\omega_N \rho^N}F_-(s){\rm d}s\right){\rm d}\rho,\ \ x\in [0,b],
				\end{equation}
		respectively. The following comparison principle holds: 
		
		\begin{proposition} \label{talenti-result-0}
			Let  $V_+$ and $V_-$ from {\rm (\ref{v-definit})} and {\rm (\ref{w-definit})}, respectively. Then
			\begin{equation}\label{1-compar-0}
				{\sf AVR}_ {\sf m}^\frac{2}{N}	u_+^*\leq V_+\ \ {in}\ \ [0,a];
			\end{equation}
			\begin{equation}\label{2-compar-0}
				{\sf AVR}_ {\sf m}^\frac{2}{N}	u_-^*\leq V_-\ \ {in}\ \ [0,b].
			\end{equation}
		\end{proposition}
		\begin{proof}
			We  prove \eqref{1-compar-0}, the estimate \eqref{2-compar-0} being similar. Keeping the previous notations, we notice that in the 1-dimensional model space 
			$([0,\infty),{\sf d}_{\rm eu},\sigma_N=N\omega_Nr^{N-1}\mathcal L^1)$, one has   that 
			\begin{equation}\label{eq:model-space-equal}
					{\rm Per}(\{u_+^*>t\})=N\omega_N^\frac{1}{N}\mu_+(t)^\frac{N-1}{N}=N\omega_N r_t^{N-1}, \quad \text{for every $t>0$};
			\end{equation}
				thus, by the isoperimetric inequality \eqref{eqn-isoperimetric-2} it follows that 
			\begin{equation}\label{Area-Area-0}
				{\sf AVR}_ {\sf m}^\frac{1}{N}	{\rm Per}(\{u_+^*>t\})\leq {\rm Per}(\{u_+>t\}),  \ \ {\rm for\  every}\ t>0.
			\end{equation}
			By \eqref{Area-Area-0} and Propositions \ref{isoperi-prop} and \ref{f-hasonlitas}, one has for $\mathcal L^1$-a.e.\ $t>0$ that
			$${\sf AVR}_ {\sf m}^\frac{2}{N}	{\rm Per}^2(\{u_+^*>t\})\leq \mu_+'(t)\int_{\{u_+>t\}}\Delta u\,{\rm d}{\sf m}\leq  -\mu_+'(t)\int_0^{\mu_+(t)} F_+(s){\rm d}s.$$ Since 
			$\mu_+'(t)=N\omega_N r_t^{N-1}r_t',
			$
			the previous relation reads
			$${\sf AVR}_ {\sf m}^\frac{2}{N} N\omega_N\leq -r_t'r_t^{1-N}\int_0^{\omega_N r_t^N} F_+(s)\,{\rm d}s, \quad \text{ for ${\mathcal L}^1$-a.e.\ $t>0$.}$$
			By an integration, it follows  that
			$${\sf AVR}_ {\sf m}^\frac{2}{N} N\omega_N\eta\leq -\int_0^\eta r_t'r_t^{1-N}\int_0^{\omega_N r_t^N} F_+(s)\, {\rm d}s{\rm d}t, \quad \text{for every $\eta\in [0,{\rm ess\,sup}\, u_+]$}.$$
			If we change the variable $r_t=\rho$ and we recall that $r_0=a$, it yields
			$${\sf AVR}_ {\sf m}^\frac{2}{N}\eta\leq \frac{1}{N\omega_N}\int_{r_\eta}^a \rho^{1-N}\left(\int_0^{\omega_N \rho^N} F_+(s)\, {\rm d}s\right){\rm d}\rho.$$
			If $r_\eta=x\in [0,a]$, one has that $\mu_+(\eta)=\omega_N r_\eta^N=\omega_N x^N$; moreover, since $u_+^\#$ is the generalized inverse of $\mu_+$, it follows that 
			$$u_+^*(x)=u_+^\#(\omega_N x^N)=u_+^\#(\mu_+(\eta))=\eta.$$ Therefore, for every $x\in [0,a]$, from \eqref{v-definit} we infer that 
			$${\sf AVR}_ {\sf m}^\frac{2}{N} u_+^*(x)\leq \frac{1}{N\omega_N}\int_{x}^a \rho^{1-N}\left(\int_0^{\omega_N \rho^N} F_+(s){\rm d}s\right){\rm d}\rho=V_+(x),$$
			which is the required relation \eqref{1-compar-0}.
		\end{proof}
		
		The following result summarizes the constructions in this section; to state it, we consider the 1-dimensional Laplacian for the function $V:[0,\infty)\to \mathbb R$ in the model space $([0,\infty),{\sf d}_{\rm eu},\sigma_N)$, i.e., 
		\begin{equation}\label{def-1-laplace}
			\Delta_{0,N} V(r)=V''(r)+\frac{N-1}{r}V'(r),\ \ r>0.
		\end{equation}

		\begin{theorem} \label{talenti-result-thm}
			Let  $V_+$ and $V_-$ from {\rm (\ref{v-definit})} and {\rm (\ref{w-definit})}, respectively. Then
			\begin{equation}\label{u-v-w}
				{\sf AVR}_ {\sf m}^\frac{4}{N}	\int_{\tilde \Omega}u^2 \, {\rm d}{\sf m}\leq \int_{0}^aV_+^2 \, {\rm d}\sigma_N+\int_{0}^b V_-^2 \, {\rm d}\sigma_N,
			\end{equation}
			and
			\begin{equation}\label{laplace-hasonlitas}
				\int_{\tilde \Omega} (\Delta u)^2 \, {\rm d}{\sf m}=\int_{0}^a (\Delta_{0,N}V_+)^2 \, {\rm d}\sigma_N+ \int_{0}^b (\Delta_{0,N}V_-)^2 \, {\rm d}\sigma_N.
			\end{equation}
		\end{theorem}
		\begin{proof}
			The proof of 	\eqref{u-v-w} follows by Proposition \ref{talenti-result-0} combined with \eqref{equimeasure}; indeed, one has that 
			\begin{eqnarray*}
				{\sf AVR}_ {\sf m}^\frac{4}{N}	\int_{\tilde \Omega}u^2 {\rm d}{\sf m}&=&{\sf AVR}_ {\sf m}^\frac{4}{N}	\int_{\Omega_+}u_+^2 {\rm d}{\sf m}+{\sf AVR}_ {\sf m}^\frac{4}{N}	\int_{ \Omega_-}u_-^2 {\rm d}{\sf m}\\&=&{\sf AVR}_ {\sf m}^\frac{4}{N}	\int_{0}^a (u_+^*)^2 {\rm d}\sigma_N+{\sf AVR}_ {\sf m}^\frac{4}{N}	\int_{ 0}^b(u_-^*)^2 {\rm d}\sigma_N\\&\leq&\int_{0}^aV_+^2 {\rm d}\sigma_N+\int_{0}^b V_-^2 {\rm d}\sigma_N.
			\end{eqnarray*}
			Let us focus now on \eqref{laplace-hasonlitas}. First, due to \eqref{simult=0} and \eqref{equimeasure}, we observe that
			\begin{align*}
				\int_{0}^{{\sf m}(\tilde \Omega)} F^2_+(s){\rm d}s
				&=\int_{0}^{{\sf m}(\tilde \Omega)} \left[(\Delta u)^\#_-(s)^2+(\Delta u)^\#_+({\sf m}(\tilde \Omega)-s)^2-2(\Delta u)^\#_-(s)(\Delta u)^\#_+({\sf m}(\tilde \Omega)-s)\right]{\rm d}s
				\\
				&=\int_{0}^{\omega_N \tilde R^N} \left[(\Delta u)^\#_-(s)^2+(\Delta u)^\#_+(s)^2\right]{\rm d}s\\&=\int_{0}^{\tilde R} \left[(\Delta u)^*_-(t)^2+(\Delta u)^*_+(t)^2\right]{\rm d}\sigma_N(t)	=\int_{\tilde \Omega} \left[(\Delta u)_-^2+(\Delta u)^2_+\right]{\rm d}{\sf m}
				\\&=\int_{\tilde \Omega} (\Delta u)^2{\rm d}{\sf m}.
			\end{align*}
			On the other hand, by using \eqref{def-1-laplace}, 
			a simple computation shows that
			\begin{equation}\label{laplace-1}
				\Delta_{0,N} V_+(r)=-F_+(\omega_N r^N),\ \ r\in (0,a],
			\end{equation}
			\begin{equation}\label{laplace-2}
				\Delta_{0,N} V_-(r)=-F_-(\omega_N r^N),\ \ r\in (0,b].
			\end{equation}
			Therefore, by \eqref{laplace-1},  \eqref{laplace-2} and  \eqref{definition-F-pm}, combined with the fact that $a^N+b^N=\tilde R^N$, one has that
			\begin{eqnarray*}
				\int_{0}^a (\Delta_{0,N}V_+)^2{\rm d}\sigma_N+ \int_{0}^b (\Delta_{0,N}V_-)^2{\rm d}\sigma_N&=&\int_{0}^a F_+^2(\omega_N r^N){\rm d}\sigma_N(r)+ \int_{0}^b F^2_-(\omega_N r^N){\rm d}\sigma_N(r)\\&=& \int_{0}^{\omega_N a^N} F_+^2(s){\rm d}s+ \int_{0}^{\omega_N b^N} F_-^2(s){\rm d}s\\&=&\int_{0}^{\omega_N \tilde R^N} F_+^2(s){\rm d}s.
			\end{eqnarray*}
			Since ${\sf m}(\tilde \Omega)=\omega_N \tilde R^N$, by combining the above findings, relation  \eqref{laplace-hasonlitas} yields at once. 
		\end{proof}
		
		In the sequel, we collect the boundary conditions in terms of the functions $V_+$ and $V_-$:

		\begin{proposition} We have the following boundary conditions: 
			\begin{itemize}
				\item[(i)] $V_+(a)=V_-(b)=0;$ 
				\item[(ii)] $V_+'(a)a^{N-1}=V_-'(b)b^{N-1};$
				\item[(iii)] $\Delta_{0,N}V_+(a)+\Delta_{0,N}V_-(b)=0.$
			\end{itemize}
		\end{proposition}
		
		\begin{proof}
			Properties from (i) directly follow by the definitions of $V_+$ and $V_-$, see \eqref{v-definit} and \eqref{w-definit}, respectively. Property (iii) also follows by \eqref{laplace-1},  \eqref{laplace-2} and  \eqref{definition-F-pm}, where we use  $a^N+b^N=\tilde R^N$.
			
			We now prove (ii). 	By using Proposition \ref{prop-laplace=00} and the fact that    $u=0$ outside  $\tilde{\Omega}$,  a similar argument as in Theorem \ref{talenti-result-thm} implies that
			\begin{eqnarray*}
				0&=&-\int_{ \Omega}\Delta u\,{\rm d}{\sf m}=-\int_{\tilde \Omega}\Delta u\,{\rm d}{\sf m}\\&=&\int_{\tilde \Omega}(\Delta u)_-{\rm d}{\sf m}-\int_{\tilde \Omega}(\Delta u)_+{\rm d}{\sf m}=\int_{0}^{\tilde R}(\Delta  u)_-^*{\rm d}\sigma_N -\int_{0}^{\tilde R}(\Delta  u)_+^*{\rm d}\sigma_N\\&=&\int_{0}^{\tilde R}(\Delta  u)_-^\#(\omega_N t^N){\rm d}\sigma_N(t)-\int_{0}^{\tilde R}(\Delta  u)_+^\#(\omega_N t^N){\rm d}\sigma_N(t)\\&=&
				\int_0^{\omega_N \tilde R^N}(\Delta u)_-^\#(s){\rm d}s-\int_0^{\omega_N \tilde R^N}(\Delta u)_+^\#(s){\rm d}s\\&=&\int_0^{\omega_N \tilde R^N}(\Delta u)_-^\#(s){\rm d}s-\int_0^{\omega_N \tilde R^N}(\Delta u)_+^\#(\omega_N \tilde R^N-s){\rm d}s\\&=&\int_0^{\omega_N \tilde R^N}F_+(s){\rm d}s = \int_0^{\omega_N a^N}F_+(s){\rm d}s - \int_0^{\omega_N b^N}F_-(s){\rm d}s\\&=& \int_0^{a}F_+(\omega_N t^N){\rm d}\sigma_N(t) - \int_0^{b}F_-(\omega_N t^N){\rm d}\sigma_N(t)\\&=& -\int_0^{a}\Delta_{0,N} V_+(t){\rm d}\sigma_N(t) + \int_0^{b}\Delta_{0,N} V_-(t){\rm d}\sigma_N(t).
			\end{eqnarray*}
			Since a direct computation gives that $$\int_0^{a}\Delta_{0,N} V_+(t){\rm d}\sigma_N(t)=N\omega_N V_+'(a)a^{N-1}\ \ {\rm and}\ \ \ \int_0^{b}\Delta_{0,N} V_-(t){\rm d}\sigma_N(t)=N\omega_N V_-'(b)b^{N-1},$$
			the claim from (ii) follows. 
		\end{proof}
		
		We summarize the results from this section. According to \eqref{Lieb-modified}, by Theorem \ref{talenti-result-thm} we have that
		\begin{equation}\label{Lieb-modified-final}
				\Lambda_{\sf m}(\Omega)\geq \frac{\displaystyle\int_{\tilde \Omega} (\Delta u)^2 {\rm d}{\sf m}}{\displaystyle\int_{\tilde \Omega} u^2 {\rm d}{\sf m}}\geq {\sf AVR}_ {\sf m}^\frac{4}{N}\frac{\displaystyle\int_{0}^a (\Delta_{0,N}V_+)^2{\rm d}\sigma_N+ \int_{0}^b (\Delta_{0,N}V_-)^2{\rm d}\sigma_N }{\displaystyle\int_{0}^aV_+^2 {\rm d}\sigma_N+\int_{0}^b V_-^2 {\rm d}\sigma_N}\geq {\sf AVR}_ {\sf m}^\frac{4}{N}\mathcal J^{a,b}_N,
		\end{equation}
		where 
		\begin{equation}\label{J_a_b}
			\mathcal J^{a,b}_N:=\inf_{U,W} \frac{\displaystyle\int_{0}^a (\Delta_{0,N}U)^2{\rm d}\sigma_N+ \int_{0}^b (\Delta_{0,N}W)^2{\rm d}\sigma_N }{\displaystyle\int_{0}^aU^2 {\rm d}\sigma_N+\int_{0}^b W^2 {\rm d}\sigma_N},
			\end{equation}
		and $a^N+b^N=\tilde R^N$ with $\omega_N\tilde R^N={\sf m}(\tilde \Omega)\leq {\sf m}(\Omega)$, while the infimum in the coupled minimization problem \eqref{J_a_b} is considered for functions 
		$U:[0,a]\to [0,\infty)$ and $W:[0,b]\to [0,\infty)$  subject to the boundary 
		conditions 
		\begin{equation}\label{BCs}
			\ds\left\{ \begin{array}{lll}
				U(a)=W(b)=0; \\
				U'(a)a^{N-1}=W'(b)b^{N-1};\\
				\Delta_{0,N}U(a)+\Delta_{0,N}W(b)=0.
			\end{array}\right.
		\end{equation}
		Slightly modifying the proof of Proposition \ref{prop-inf-min} shows that the infimum in \eqref{J_a_b} is achieved.  
		
		The goal of the next section is to prove that for every $a,b\geq 0$ with $a^N+b^N=\tilde R^N$, one has that
		\begin{equation}\label{reduction}
		\mathcal 	J^{a,b}_N\geq \mathcal J^{\tilde R,0}_N = \mathcal J^{0,\tilde R}_N
		\end{equation}
		for a \textit{certain range} of $N$, i.e., we can algebraically cancel one of the parameters $a$ and $b$. Note that in the limit case one has
		\begin{equation}\label{minimizer}
			\mathcal J^{\tilde R,0}_N=\min_{U\in W^{2,2}_0((0,\tilde R),\sigma_N)} \frac{\displaystyle \int_{0}^{\tilde R} (\Delta_{0,N}U)^2{\rm d}\sigma_N }{\displaystyle\int_{0}^{\tilde R} U^2 {\rm d}\sigma_N},
		\end{equation}
		thus, if $\mathcal J^{\tilde R,0}_N=h^4$, the Euler--Lagrange equation gives 
		$$\Delta_{0,N}^2 U=h^4 U \ {\rm in} \ \  [0,\tilde R],$$
		subject to the clamped boundary conditions $U(\tilde R)=U'(\tilde R)=0.$  According to \eqref{def-1-laplace}, the classical solution of the latter fourth order equation is given by 
		$$U(s)=s^{-\nu}\left(A_1J_\nu(hs)+A_2I_\nu(hs)+A_3Y_\nu(hs)+A_4K_\nu(hs)\right),\ s\in(0,\tilde R],$$
		where $A_i\in \mathbb R$ are some constants, while $J_\nu$ and $Y_\nu$ are the Bessel functions of first and second kind,  and $I_\nu$ and $K_\nu$ are the modified Bessel functions of first and second kind, respectively, of order $\nu=\frac{N}{2}-1.$ Since both $Y_\nu$ and $K_\nu$ have singularities at the origin, they should be canceled, thus 
	\begin{equation}\label{solution-bessel}
		U(s)=s^{-\nu}(A_1J_\nu(hs)+A_2I_\nu(hs))\ s\in(0,\tilde R].
	\end{equation}
	The boundary conditions $U(\tilde R)=U'(\tilde R)=0$ imply that $U$ is nontrivial if \begin{equation}\label{determinant-0}
		{\rm det}\left[\begin{matrix}
			J_\nu(h\tilde R) & I_\nu(h\tilde R) \\
			J'_\nu(h\tilde R) & I'_\nu(h\tilde R)
		\end{matrix}\right]=0,
	\end{equation} 
	to which the first positive solution is $h:=\frac{h_\nu}{\tilde R},$ where $h_\nu$ is the first positive root of the cross-product   of $J_\nu$ and $I_\nu$. In particular, \eqref{solution-bessel} reduces, up to multiplicative factor, to 
	\begin{equation}\label{solution-bessel-0}
		U(s)=s^{-\nu}\left(I_\nu(h_\nu) J_\nu\left(h_\nu\frac{s}{\tilde R}\right)-J_\nu(h_\nu)I_\nu\left(h_\nu\frac{s}{\tilde R}\right)\right),\ s\in(0,\tilde R].
	\end{equation}
		
	Therefore, once we have the property \eqref{reduction}, it follows by \eqref{Lieb-modified-final} that 
		\begin{eqnarray}\label{estimate-last}
			\Lambda_{\sf m}(\Omega)\geq {\sf AVR}_ {\sf m}^\frac{4}{N}\mathcal J^{\tilde R,0}_N=
			{\sf AVR}_ {\sf m}^\frac{4}{N} \frac{h_\nu^4}{\tilde R^{4}}=
			{\sf AVR}_ {\sf m}^\frac{4}{N} h_\nu^4\left(\frac{\omega_N}{{\sf m}(\tilde \Omega)}\right)^\frac{4}{N}\geq {\sf AVR}_ {\sf m}^\frac{4}{N} h_\nu^4\left(\frac{\omega_N}{{\sf m}(\Omega)}\right)^\frac{4}{N},
		\end{eqnarray}
		which is precisely the required inequality \eqref{inequality}. 
		
		\section{Reducing the coupled minimization: proof of Theorem \ref{theorem-main}}

		The argument in this section is closely related to the one by Ashbaugh and Benguria \cite{A-B-Duke}, where a careful analysis of Bessel functions is performed. In fact, we want to determine a wider range of not necessarily integers $N>1$  where the aforementioned reduction \eqref{reduction} still works in the 1-dimensional model space $([0,\infty),{\sf d}_{\rm eu},\sigma_N)$. Although we have numerical evidences for the validity of \eqref{reduction} for every $N\in (1,N_0)$ with $N_0\approx 3.30417$ (see the argument in \eqref{N-0-def}), we face  serious technical difficulties by handling peculiar properties of Bessel functions. However, fine estimates show  the validity of \eqref{reduction} when $N$ is  close enough to the integer dimensions $2$ and $3$. In the sequel, we shall focus on this aspect;    for the reader's convenience   we recall the main steps  from \cite{A-B-Duke} adapted to our 1-dimensional model space. 
		
	\subsection{Ashbaugh--Benguria's argument in the 1-dimensional model space}	By \eqref{Lieb-modified-final}, we have following variational problem  
		$$\mathcal J^{a,b}_N=\min_{U,W} \frac{\displaystyle\int_{0}^a (\Delta_{0,N}U)^2{\rm d}\sigma_N+ \int_{0}^b (\Delta_{0,N}W)^2{\rm d}\sigma_N }{\displaystyle\int_{0}^aU^2 {\rm d}\sigma_N+\int_{0}^b W^2 {\rm d}\sigma_N}=:h^4>0$$ for some $h=h_N(a)>0,$
		in the 1-dimensional model space $([0,\infty),{\sf d}_{\rm eu},\sigma_N)$, where $a^N+b^N=\tilde R^N$ with $\omega_N\tilde R^N={\sf m}(\tilde \Omega)\leq {\sf m}(\Omega)$, and the infimum is considered for functions 
		$U:[0,a]\to [0,\infty)$ and $W:[0,b]\to [0,\infty)$  subject to the boundary 
		conditions \eqref{BCs}. 
		By the previous variational problem, the Euler--Lagrange equation provides the  system of ordinary differential equations
		$$\ds\left\{ \begin{array}{ccc}
			\Delta_{0,N}^2 U=h^4 U & {\rm in} & [0,a]; \\
			\Delta_{0,N}^2 W=h^4 W & {\rm in} &  [0,b],
		\end{array}\right.$$
		subject to the boundary conditions \eqref{BCs}.  As before, one has that
		$$U(s)=s^{-\nu}(A_1J_\nu(hs)+A_2I_\nu(hs)),\ s\in(0,a],$$
		$$W(s)=s^{-\nu}(B_1J_\nu(hs)+B_2I_\nu(hs)),\ s\in(0,b],$$
		for some constant $A_i,B_i\in \mathbb R$, $i\in \{1,2\}.$ According to the boundary conditions \eqref{BCs} and basic recurrence relations for Bessel functions, we have that 
		$$\ds\left\{ \begin{array}{lll}
			A_1J_\nu(ha)+A_2I_\nu(ha)=0; \\
			B_1J_\nu(hb)+B_2I_\nu(hb)=0;\\
			(-A_1J_{\nu+1}(ha)+A_2I_{\nu+1}(ha))a^{\nu+1}=(-B_1J_{\nu+1}(hb)+B_2I_{\nu+1}(hb))b^{\nu+1};\\
			(-A_1J_\nu(ha)+A_2I_\nu(ha))a^{-\nu}+(-B_1J_\nu(hb)+B_2I_\nu(hb))b^{-\nu}=0.
		\end{array}\right.
		$$
		In order not to have a trivial solution to the latter system in $A_i,B_i\in \mathbb R$, $i\in \{1,2\}$, it is necessary that
		\begin{equation}\label{determinant}
			{\rm det}\left[\begin{matrix}
				J_\nu(ha) & I_\nu(ha) & 0 & 0\\
				0 & 0 &J_\nu(hb) & I_\nu(hb)\\
				-J_{\nu+1}(ha)a^{\nu+1} & I_{\nu+1}(ha)a^{\nu+1} &J_{\nu+1}(hb)b^{\nu+1} &-I_{\nu+1}(hb)b^{\nu+1}\\
				-J_{\nu}(ha)a^{-\nu} & I_{\nu}(ha)a^{-\nu} &-J_{\nu}(hb)b^{-\nu} &I_{\nu}(hb)b^{-\nu}
			\end{matrix}\right]=0.
		\end{equation}
		If we introduce the notation 
		\begin{equation}\label{determinant-2}
			\mathcal K_\nu(s)=s^{2\nu+1}\left(\frac{J_{\nu+1}}{J_\nu}(s)+\frac{I_{\nu+1}}{I_\nu}(s)\right),
		\end{equation}
		equation \eqref{determinant} can be written into its equivalent form 
		\begin{equation}\label{determinant-3}
			\mathcal K_\nu(ha)+\mathcal K_\nu(hb)=0.
		\end{equation}
		Without loss of generality, by using a usual scaling, we may assume that $\tilde R=1$, i.e., $a^N+b^N=1.$ Moreover, the latter equation being  symmetric in $a$ and $b$, we may also assume that $0<a\leq b<1.$ If $h=h_\nu(a)$
		denotes the lowest positive solution of \eqref{determinant-3}, 
		our purpose is to prove that 
		\begin{equation}\label{expected-ineq}
			h_\nu(a)> h_\nu,
		\end{equation}
		where $h_\nu>0$ corresponds to $h_\nu(0)$, which  is the first positive root of $K_\nu,$ being  the first positive root of the cross-product of the Bessel functions $J_\nu$ and $I_\nu$; indeed, the validity of  \eqref{expected-ineq} directly implies \eqref{reduction} even with strict inequality when $0<a\leq b<1,$ thus also the validity of \eqref{inequality}. 
		
		Due to basic properties of  $J_\nu$, the function $\mathcal K_\nu$ has simple poles, denoted by $j_{\nu,k}$, $k\in \mathbb N$. 
		We first observe that $\mathcal K_\nu$ is increasing between any two consecutive poles, see \cite{A-B-Duke}, 
		and according to the limiting properties 
		\begin{equation}\label{limiting-J-I}
			J_\nu(s)\sim \left(\frac{s}{2}\right)^\nu\frac{1}{\Gamma(\nu+1)},\ \ I_\nu(s)\sim \left(\frac{s}{2}\right)^\nu\frac{1}{\Gamma(\nu+1)}\ \ {\rm as}\ s\to 0,
		\end{equation}
		see relations (10.7.3) and (10.30.1) from \cite{Olver-etal}, we obtain that $K_\nu(0)=0.$ In particular, it follows that no positive zero exists in the interval $(0,j_{\nu,1}/b)$ for the function $h\mapsto \mathcal K_\nu(ha)+\mathcal K_\nu(hb)$. Moreover, due to the presence of the second term, it turns out that $h\mapsto \mathcal K_\nu(ha)+\mathcal K_\nu(hb)$ has its first pole in $j_{\nu,1}/b$. The second pole is either $j_{\nu,1}/a$ (coming from the first term) or $j_{\nu,2}/b$ (being the second pole of the second term). In conclusion, it turns out that the first positive zero $h_\nu(a)$ of $h\mapsto \mathcal K_\nu(ha)+\mathcal K_\nu(hb)$ is situated between $j_{\nu,1}/b$ and $\min\{j_{\nu,1}/a,j_{\nu,2}/b\}$. In the limiting case when $a$ and $b$ approach each other ($a\nearrow 2^{-1/N}$ and $b\searrow 2^{-1/N}$), the two poles come arbitrarily close to $h_\nu(2^{-1/N})$, thus $$h_\nu(2^{-1/N})=j_{\nu,1}/2^{-1/N}.$$
		Therefore, in order to have \eqref{expected-ineq}, we should have $$2^{1/N}j_{\nu,1}\geq h_\nu.$$
		Recalling that $\nu=\frac{N}{2}-1,$ we consider 
		\begin{equation}\label{N-0-def}
			N_0:=\sup\left\{N>1:2^{1/N}j_{\nu,1}> h_\nu\right\}.	
		\end{equation}
		Numerical arguments show that $N_0\approx 3.30417$. By \eqref{N-0-def}, it follows that the present argument does not work beyond the dimension $N_0.$ 
			
		First, if we fix $N\in (0,N_0)$, one has that $2^{1/N}j_{\nu,1}> h_\nu $. In particular, we have  $h_N(a)> h_\nu $, i.e., \eqref{expected-ineq} holds, whenever $h_\nu\leq j_{\nu,1}/b$ (the value $j_{\nu,1}/b$ being the first pole for $h\mapsto \mathcal K_\nu(ha)+\mathcal K_\nu(hb)$). Therefore, since $a^N+b^N=1$, it follows that \eqref{expected-ineq} holds for  $b\leq j_{\nu,1}/h_\nu$, which is equivalent to $a\in [(1-(j_{\nu,1}/h_\nu)^N)^{1/N},2^{-1/N}].$ 
		
		For the complement range of $a$, property \eqref{expected-ineq} also holds  once  we have that 
		\begin{equation}\label{ineq:K_a_b}
			\mathcal K_\nu(h_\nu a)+\mathcal K_\nu(h_\nu b)<0,\ \ \forall a\in (0,(1-(j_{\nu,1}/h_\nu)^N)^{1/N}).
		\end{equation}
		 Indeed, since $t\mapsto \mathcal K_\nu(t a)+\mathcal K_\nu(t b)$ is strictly increasing between consecutive poles, by \eqref{ineq:K_a_b} we would have $\mathcal K_\nu(h_\nu a)+\mathcal K_\nu(h_\nu b)<0=\mathcal K_\nu(h_\nu(a) a)+\mathcal K_\nu(h_\nu(a) b)$, thus $h_\nu<h_\nu(a)$, which is exactly the expected inequality \eqref{expected-ineq} for the remaining case $a\in (0,(1-(j_{\nu,1}/h_\nu)^N)^{1/N})$. 
		 
	\subsection{Concluding the proof of inequality \eqref{inequality}}	 In order to prove inequality \eqref{inequality}, it remains to study the validity of \eqref{ineq:K_a_b}, as the latter implies \eqref{expected-ineq}. We recall that this is known for $N\in \{2,3\},$ see \cite[Appendix 1]{A-B-Duke}.  In order to have \eqref{ineq:K_a_b} for a wider range of $N$, we proceed similarly as in \cite{A-B-Duke}; although it is expected the validity of \eqref{ineq:K_a_b} for every $N\in (1,N_0)$, there are serious technical obstructions as we see below. In present paper, we extend this property for dimensions $N$ close to 2 and 3, respectively. To do this,  by using Mittag--Leffler representations for the ratios  $J_{\nu+1}/J_\nu$ and $I_{\nu+1}/I_\nu$, we recall that the validity of \eqref{ineq:K_a_b} is guaranteed whenever we prove that 
		 \begin{equation}\label{A-B-nu}
		 A(\nu,a) + B(\nu,a)<0,\ \ \forall a\in (0,(1-(j_{\nu,1}/h_\nu)^N)^{1/N}),
		 \end{equation}
		  where 
		  \begin{equation}\label{A-nu}
		  	A(\nu,a)=\frac{h_\nu^4(ab)^{4-N}-j_{\nu,1}^4(a^{4-N}+b^{4-N})}{(j_{\nu,1}^4-(h_\nu a)^4)((h_\nu b)^4-j_{\nu,1}^4)},
		  \end{equation}
		and 
	 \begin{equation}\label{B-nu}
		B(\nu,a)=a^{4-N}\sum_{k\geq 2}\frac{1}{j_{\nu,k}^4-(h_\nu a)^4}+b^{4-N}\sum_{k\geq 2}\frac{1}{j_{\nu,k}^4-(h_\nu b)^4},
	\end{equation}
see \cite[rels.\ (45)\&(46)]{A-B-Duke}, with $a^N+b^N=1$. 

On the one hand, due to the fact that $\nu\mapsto j_{\nu,1}$ and $\nu\mapsto h_\nu$ are continuous functions, and $A$ is also continuous on $D=\left\{(\nu,a)\in \mathbb R^2: \nu\in (-\frac{1}{2},\frac{N_0}{2}-1),a\in [0,(1-(j_{\nu,1}/h_\nu)^N)^{1/N}]\right\}$,  an elementary argument shows that the function $\alpha:(-\frac{1}{2},\frac{N_0}{2}-1)\to \mathbb R$ given by 
$$\alpha(\nu):=\max \{A(\nu,a):a\in [0,(1-(j_{\nu,1}/h_\nu)^N)^{1/N}]\}$$
is upper-semicontinuous. 
On the other hand, if $\delta_\nu:=h_\nu^4/j_{\nu,2}^2$, then for every $k\geq 2$ and $0\leq a,b\leq 1,$ one has that $j_{\nu,k}^4-(h_\nu a)^4\geq j_{\nu,k}^4(1-\delta_\nu)$ and  $j_{\nu,k}^4-(h_\nu b)^4\geq j_{\nu,k}^4(1-\delta_\nu)$, respectively. Therefore, 
$$B(\nu,a)\leq \frac{a^{4-N}+b^{4-N}}{1-\delta_\nu}\sum_{k\geq 2}\frac{1}{j_{\nu,k}^4}.$$ 
By the Rayleigh sum $\sum_{k\geq 1}\frac{1}{j_{\nu,k}^4}=\frac{1}{16(\nu+1)^2(\nu+2)}$, see e.g.\ Watson \cite{Watson}, and the fact that $a^{4-N}+b^{4-N}\leq 2^{2-\frac{4}{N}}$ (since $a^N+b^N=1$), it follows that 
$$B(\nu,a)\leq \frac{2^{2-\frac{4}{N}}}{1-\delta_\nu}\left(\frac{1}{16(\nu+1)^2(\nu+2)}-\frac{1}{j_{\nu,1}^4}\right)\equiv \beta(\nu).$$
Clearly, $\beta$ is continuous in $(-\frac{1}{2},\frac{N_0}{2}-1)$. 		

	We recall by Ashbaugh and Benguria \cite[Appendix 1]{A-B-Duke} that $\alpha(0)+\beta(0)\leq -0.00158$ (corresponding to $N=2$) and 	$\alpha(\frac{1}{2})+\beta(\frac{1}{2})\leq -0.000417$ (corresponding to $N=3$); therefore, the upper-semicontinuity of $\alpha+\beta$ in $(-\frac{1}{2},\frac{N_0}{2}-1)$ implies the existence of $\varepsilon_0>0$ such that $\alpha(\nu)+\beta(\nu)<0$ for every $\nu\in \left(-\frac{\varepsilon_0}{2},\frac{\varepsilon_0}{2}\right)\cup \left(\frac{1-\varepsilon_0}{2},\frac{1+\varepsilon_0}{2}\right)$, which corresponds to $N\in (2-\varepsilon_0,2+\varepsilon_0)\cup (3-\varepsilon_0,3+\varepsilon_0)$. In particular, this concludes the proof of \eqref{A-B-nu}.  
	
	\begin{remark}\label{remark-dimension}\rm 
		As we already stated, numerical tests show the validity of \eqref{A-B-nu} for \textit{every} $N\in (1,N_0)$ with $N_0\approx 3.30417$, suggesting the possibility to extend  Theorem \ref{theorem-main} to the whole range $1<N<N_0.$
		However, the proof of this claim is  non-trivial, due to the presence of the roots of Bessel functions and their cross-product. In fact, once we are closer and closer to $N_0$, we need finer and finer \textit{explicit} estimates for  $j_{\nu,1}$, $j_{\nu,2}$ and $h_\nu$, respectively, which are not available in the literature (mainly, for $h_\nu$). We believe such a problem might  be interesting for experts working in the theory of special functions.  
	\end{remark}
		
		\section{Sharpness and rigidity in the main inequality \eqref{inequality}: proof of Theorem \ref{theorem-main-sharpness-rigidity}}		
		
		\subsection{Sharpness: proof of Theorem \ref{theorem-main-sharpness-rigidity}/(i)}\label{subsection-sharp}
		$\phantom{}$\smallskip
		
\noindent In order to prepare the proof,  we consider the function  
		\begin{equation}\label{f-0-def}
			f_0(s)=s^{-\nu}(AJ_\nu(s)+BI_\nu(s)),\ \ s\in [0,h_\nu],
		\end{equation}
		where $\nu=\frac{N}{2}-1$ and $B=-A\frac{J_\nu(h_\nu)}{I_\nu(h_\nu)}.$ In fact, at the origin, by using the limiting properties \eqref{limiting-J-I}, the function $f_0$ is extended by continuity as $f_0(0)=\frac{A+B}{2^\nu\Gamma(\nu+1)}$. 
		In addition, by construction, one has that $f_0(h_\nu)=f_0'(h_\nu)=0.$ In terms of the function $f_0$, one has the  following identity
		\begin{equation}\label{identity-bessel}
			\ds\int_0^1 \left[\frac{(N-1)^2}{h_\nu^2t^2}{f_0'}^2\left(h_\nu t\right)-2(f_0'f_0'')'\left(h_\nu t\right)+{f_0''}^2\left(h_\nu t\right)\right]t^{N-1}{\rm d}t = \int_0^1f_0^2\left(h_\nu t\right)t^{N-1}{\rm d}t,
		\end{equation}
		which can be checked either by direct computation or by the  fact that $f_0$ is an extremizer in the $1$-dimensional clamped plate  problem, i.e., in particular, 
		\begin{equation}\label{laplace-f-0}
				\frac{\displaystyle \int_{0}^{h_\nu} (\Delta_{0,N} f_0)^2{\rm d}\sigma_N }{\displaystyle\int_{0}^{h_\nu} f_0^2 {\rm d}\sigma_N}=
			1.
		\end{equation}
	We now divide the proof into two steps. 
		
		\textbf{Step 1}. {\it Equality in \eqref{inequality} is achieved for metric balls in metric measure cones.} 
	We assume that $X$ is isometric to an $N$-Euclidean metric measure cone and $\Omega$ is isometric to a metric ball $B_r(\overline x)$ for some $r>0$, where $\overline x$ is one of the tips of $X$. We claim that
	\begin{equation}\label{cone-equality}
	   \Lambda_{\sf m}\left(\Omega\right)={\sf AVR}_ {\sf m}^\frac{4}{N}h_\nu^4\left(\frac{\omega_N}{{\sf m}(\Omega)}\right)^\frac{4}{N}. 
	\end{equation}
Since $X$ is isometric to an $N$-Euclidean metric measure cone, having the tip $\overline x$, it turns out that \begin{equation}\label{cone-balls}
    {\sf m}(B_s(\overline x))={\sf AVR}_ {\sf m}\omega_N s^{N}\ \ {\rm  and}\ \ {\sf m}^+(B_s(\overline x))={\sf AVR}_ {\sf m}N\omega_N s^{N-1},\ \ {\rm for\ every}\ s\in [0,r],
\end{equation}   and
	$$\Delta \rho =\frac{N-1}{\rho}\ \ {\rm on}\ \ X\setminus \{\overline x\},$$  where $\rho={\sf d}(\overline x,\cdot)$, see e.g.\ De Philippis and Gigli \cite[Proposition 3.7]{DeP-G}. In particular, by \eqref{cone-balls} one has that $r={\sf AVR}_ {\sf m}^{-{1}/{N}}R,$ where ${\sf m}(\Omega)=\omega_NR^N.$

	To check \eqref{cone-equality}, let $u:B_{{\sf AVR}_ {\sf m}^{-{1}/{N}}R}(\overline x)\to \mathbb R$ be the function given by 
 \begin{equation}\label{extremal-candidate}
    u(x)=U\left({\sf AVR}_ {\sf m}^{\frac{1}{N}}{\sf d}(\overline x,x)\right),\ \  x\in B_{{\sf AVR}_ {\sf m}^{-{1}/{N}}R}(\overline x)\simeq \Omega,  
 \end{equation}
 where $U$ is defined by
	\begin{equation}\label{U-definition}
		U(s)=s^{-\nu}\left(I_\nu(h_\nu) J_\nu\left(h_\nu\frac{s}{ R}\right)-J_\nu(h_\nu)I_\nu\left(h_\nu\frac{s}{ R}\right)\right),\ s\in(0, R].
\end{equation}
	  Let us observe that $u\in H_0^{2,2}(\Omega)$ and
	   $$  u(x)=c_0f_0(C{\sf d}(\overline x,x)),\ \ x\in B_{{\sf AVR}_ {\sf m}^{-{1}/{N}}R}(\overline x),
	$$
where $f_0$ is from \eqref{f-0-def}, while 
	$c_0={I_\nu(h_\nu)}\left(\frac{h_\nu}{R}\right)^\nu$ and $C={\sf AVR}_ {\sf m}^{\frac{1}{N}}\frac{h_\nu}{R}$.

	By using these properties, it follows that
	\begin{eqnarray*}
		\Lambda_{\sf m}\left(B_{{\sf AVR}_ {\sf m}^{-{1}/{N}}R}(\overline x)\right)&\leq& \frac{\displaystyle\int_{B_{{\sf AVR}_ {\sf m}^{-{1}/{N}}R}(\overline x)} (\Delta u)^2 {\rm d}{\sf m}}{\displaystyle\int_{B_{{\sf AVR}_ {\sf m}^{-{1}/{N}}R}(\overline x)} u^2 {\rm d}{\sf m}}=C^4	\frac{\displaystyle \int_{0}^{h_\nu} (\Delta_{0,N} f_0)^2{\rm d}\sigma_N }{\displaystyle\int_{0}^{h_\nu} f_0^2 {\rm d}\sigma_N}\\&=&C^4={\sf AVR}_ {\sf m}^\frac{4}{N}h_\nu^4\left(\frac{\omega_N}{{\sf m}(B_{{\sf AVR}_ {\sf m}^{-{1}/{N}}R}(\overline x))}\right)^\frac{4}{N}\qquad \   \qquad \qquad  ({\rm see}\  \eqref{laplace-f-0}) \\&\leq&\Lambda_{\sf m}\left(B_{{\sf AVR}_ {\sf m}^{-{1}/{N}}R}(\overline x)\right), \qquad \qquad \qquad \qquad \qquad \qquad \qquad ({\rm see}\  \eqref{inequality})
	\end{eqnarray*}
 which concludes the proof of \eqref{cone-equality}.  
		\smallskip
		
		\textbf{Step 2}.   {\it Sharpness of \eqref{inequality} for ${\sf RCD}(0,N)$ spaces with subharmonic distance function.}

		By contradiction,  assume there is  $C>{\sf AVR}_ {\sf m}^\frac{4}{N}$ such that for every  open bounded domain $\Omega\subset X$:
		\begin{equation}\label{inequality-0}
			\Lambda_{\sf m}(\Omega)\geq  C\,h_\nu^4\left(\frac{\omega_N}{{\sf m}(\Omega)}\right)^\frac{4}{N}.
		\end{equation}

		Let $x_0\in X$ from the statement of the theorem,  the set $\Omega:=B_L(x_0)$ for some $L>0$ and  the function $u_L:B_L(x_0)\to \mathbb R$ defined  by
		$$u_L(x)=f_0\left(h_\nu\frac{\rho(x)}{L}\right),\ \ x\in B_L(x_0),$$
		where $\rho(x):={\sf d}(x_0,x)$ and $f_0$ is from \eqref{f-0-def}. By the smoothness of $J_\nu$ and $I_\nu$, it follows that $u_L\in \mathsf{LIP}(B_L(x_0))$. In fact, we also have that $u_L\in H_0^{2,2}(B_L(x_0),{\sf m})$, which follows from the subsequent argument. First, we have by the area formula \eqref{area-formlula} that
		$$\int_{B_L(x_0)}u_L^2\,{\rm d}{\sf m}=\int_{B_L(x_0)}f_0^2\left(h_\nu\frac{\rho}{L}\right){\rm d}{\sf m}=\int_0^Lf_0^2\left(h_\nu\frac{s}{L}\right){\sf m}^+(B_s(x_0)){\rm d}s;$$  
		thus, by \eqref{Bishop-Gromov-monoton} one obtains that $u_L\in L^2(B_L(x_0),{\sf m})$. Moreover, by the monotone convergence theorem and \eqref{AVR-difinition}, the latter relation implies that \begin{equation}\label{first-limit}
			\lim_{L\to \infty}\frac{\ds\int_{B_L(x_0)}u_L^2\,{\rm d}{\sf m}}{L^N}	=\lim_{L\to \infty}\frac{\ds\int_0^1f_0^2\left(h_\nu t\right){\sf m}^+(B_{Lt}(x_0)){\rm d}t}{L^{N-1}}={\sf AVR}_ {\sf m}N\omega_N\int_0^1f_0^2\left(h_\nu t\right)t^{N-1}{\rm d}t.
		\end{equation}
		
		On the other hand, according to Gigli \cite[rel.\ (1.44)]{Gigli2}, one has ${\sf m}$-a.e.\ that
		$$\Delta u_L=\frac{h_\nu}{L}f_0'\left(h_\nu\frac{\rho}{L}\right)\Delta \rho+\frac{h_\nu^2}{L^2}f_0''\left(h_\nu\frac{\rho}{L}\right)|\nabla \rho|^2=\frac{h_\nu}{L}f_0'\left(h_\nu\frac{\rho}{L}\right)\Delta \rho+\frac{h_\nu^2}{L^2}f_0''\left(h_\nu\frac{\rho}{L}\right),$$ 
		where we used the eikonal equation \eqref{eikonal}. 
		 Therefore, we have
		\begin{eqnarray*}
			\int_{B_L(x_0)}(\Delta u_L)^2{\rm d}{\sf m}= I + II +III,
		\end{eqnarray*}
		where
		$$I:=\frac{h_\nu^2}{L^2}\int_{B_L(x_0)}{f_0'}^2\left(h_\nu\frac{\rho}{L}\right)(\Delta \rho)^2{\rm d}{\sf m},\ \ II:=2\frac{h_\nu^3}{L^3}\int_{B_L(x_0)}(f_0'f_0'')\left(h_\nu\frac{\rho}{L}\right)\Delta \rho\,{\rm d}{\sf m}$$
		and 
		$$III:=\frac{h_\nu^4}{L^4}\int_{B_L(x_0)}{f_0''}^2\left(h_\nu\frac{\rho}{L}\right){\rm d}{\sf m}.$$
		Since $(X,{\sf d},{\sf m})$ is an ${\sf RCD}(0,N)$ space, the Laplace comparison shows that $$\Delta\rho \leq \frac{N-1}{\rho}\ \  {\rm on}\ \  X\setminus \{x_0\},$$ see Gigli \cite[Corollary 5.15]{Gigli1}. By assumption,    $\rho$ is subharmonic, i.e., $\Delta\rho\geq 0$; thus, combining these facts with the co-area formula \eqref{area-formlula}, it follows that 
		$$I\leq \frac{(N-1)^2h_\nu^2}{L^2}\int_{B_L(x_0)}{f_0'}^2\left(h_\nu\frac{\rho}{L}\right)\frac{1}{\rho^2}{\rm d}{\sf m}=\frac{(N-1)^2h_\nu^2}{L^2}\int_0^L{f_0'}^2\left(h_\nu \frac{s}{L}\right)\frac{1}{s^2}{\sf m}^+(B_s(x_0)){\rm d}s.$$
		Note that the latter integral is finite, since $$\frac{f_0'(s)}{s}=s^{-\nu-1}(-AJ_{\nu+1}(s)+BI_{\nu+1}(s))\sim \frac{-A+B}{2^{\nu+1}\Gamma(\nu+2)}\ \ {\rm as}\ s\to 0.$$
		Due to the fact that $f_0'(h_\nu)=0$, we infer that 
		\begin{eqnarray*}
			II&=&2\frac{h_\nu^3}{L^3}\int_{B_L(x_0)}(f_0'f_0'')\left(h_\nu\frac{\rho}{L}\right)\Delta \rho\,{\rm d}{\sf m}=-2\frac{h_\nu^4}{L^4}\int_{B_L(x_0)}(f_0'f_0'')'\left(h_\nu\frac{\rho}{L}\right)|\nabla \rho|^2\,{\rm d}{\sf m}\\&=&-2\frac{h_\nu^4}{L^4}\int_{B_L(x_0)}(f_0'f_0'')'\left(h_\nu\frac{\rho}{L}\right){\rm d}{\sf m}=-2\frac{h_\nu^4}{L^4}\int_0^L(f_0'f_0'')'\left(h_\nu\frac{s}{L}\right){\sf m}^+(B_s(x_0)){\rm d}s.
		\end{eqnarray*}
		Collecting these computations and changing variables, it turns out that
		\begin{eqnarray*}
			\int_{B_L(x_0)}(\Delta u_L)^2{\rm d}{\sf m} &\leq& L \ds\int_0^1 \left[\frac{(N-1)^2h_\nu^2}{L^4}{f_0'}^2\left(h_\nu t\right)\frac{1}{t^2}-2\frac{h_\nu^4}{L^4}(f_0'f_0'')'\left(h_\nu t\right)\right.\\&& \qquad\quad  +\left.\frac{h_\nu^4}{L^4}{f_0''}^2\left(h_\nu t\right)\right]{\sf m}^+(B_{Lt}(x_0)){\rm d}t.  
		\end{eqnarray*}
		Thus, by the monotone convergence theorem, we get that
		$$\lim_{L\to \infty}\frac{\ds\int_{B_L(x_0)}(\Delta u_L)^2{\rm d}{\sf m}}{L^{N-4}}\leq {\sf AVR}_ {\sf m}	N\omega_N {h_\nu^4}\ds\int_0^1 \left[\frac{(N-1)^2}{h_\nu^2t^2}{f_0'}^2\left(h_\nu t\right)-2(f_0'f_0'')'\left(h_\nu t\right)+{f_0''}^2\left(h_\nu t\right)\right]t^{N-1}{\rm d}t.$$
		Combining the latter estimate with relations \eqref{inequality-0}, \eqref{identity-bessel}  and \eqref{first-limit},  we obtain that
		\begin{eqnarray*}
			C&\leq&h_\nu^{-4} \lim_{L\to \infty} \Lambda_{\sf m}(B_L(x_0))\left(\frac{{\sf m}(B_L(x_0))}{\omega_N}\right)^\frac{4}{N}\\&\leq &
			h_\nu^{-4}\lim_{L\to \infty}  \frac{\ds\int_{B_L(x_0)}(\Delta u_L)^2{\rm d}{\sf m}}{L^{N-4}}\frac{L^N}{\ds\int_{B_L(x_0)}u_L^2\,{\rm d}{\sf m}}\left(\frac{{\sf m}(B_L(x_0))}{\omega_NL^N}\right)^\frac{4}{N}\\&\leq &{\sf AVR}_ {\sf m}^\frac{4}{N},
		\end{eqnarray*}
		which contradicts our initial assumption 
		$C>{\sf AVR}_ {\sf m}^\frac{4}{N}$. \hfill $\square$
		
		\subsection{Rigidity:  proof of  Theorem \ref{theorem-main-sharpness-rigidity}/(ii)} We divide the proof into several steps.  
		\smallskip

  	\textbf{Step 1}. \textit{If equality is attained in \eqref{inequality} then $X$ is an $N$-Euclidean cone, $\Omega=B_r(\bar x)$ is a metric ball centered at a tip $\bar{x}$ of the cone structure of $X$, and  $u>0$ on $\Omega$, up to a multiplicative constant; moreover, $r={\sf AVR}_ {\sf m}^{-{1}/{N}}R,$ where ${\sf m}(\Omega)=\omega_NR^N.$}
		
		We assume that an open bounded set $\Omega\subset X$ saturates inequality \eqref{inequality}, i.e., 
		\begin{equation}\label{equality}
			\Lambda_{\sf m}(\Omega)= {\sf AVR}_ {\sf m}^\frac{4}{N}h_\nu^4\left(\frac{\omega_N}{{\sf m}(\Omega)}\right)^\frac{4}{N}. 
		\end{equation}
	As before, let $u:\Omega\to \mathbb R$ be an eigenfunction for $\Lambda_{\sf m}(\Omega)$. In particular, due to	\eqref{equality} and \eqref{estimate-last}, it follows that ${\sf m}(\tilde \Omega)={\sf m}(\Omega)$, i.e., ${\sf m}(\{u=0\})=0$, thus $\tilde R=R$. Moreover,  according to 
	 \eqref{Lieb-modified-final} and \eqref{expected-ineq}, one cannot have $0<a, b<\tilde R=R$, where $\omega_N a^N={\sf m}(\Omega_+)$ and $\omega_N b^N={\sf m}(\Omega_-)$, see \eqref{a-b}.
	 In particular, either $a=0$ or $b=0$, which implies that the eigenfunction $u:\Omega\to \mathbb R$ is sign-preserving; for simplicity, we assume that  ${\sf m}(\Omega_-)=0$, i.e., $u>0$ ${\sf m}$-a.e.\ on $\Omega$, thus $a=R$ and $b=0$.  Since, by Proposition  \ref{prop-lipschitz}, $u$ is continuous on $\Omega$, it follows that 
  \begin{equation}\label{eq:Omegau>0}
  \Omega=\{u>0\}.
  \end{equation}
		A careful analysis of the proof of \eqref{inequality} shows that equality has to hold in the isoperimetric inequality \eqref{Area-Area-0}, i.e., 
		\begin{equation}\label{eq:AVRPerEq}
			{\rm Per}(\{u>t\})= N\omega_N^\frac{1}{N}{\sf AVR}_ {\sf m}^\frac{1}{N} {\sf m}(\{u>t\})^\frac{N-1}{N}  \quad \text{for}\ \mathcal{L}^1\text{-a.e.}\ t\in [0,\max u].
		\end{equation}
		According to the characterization of the equality case in \eqref{eqn-isoperimetric-2}, it follows that
		\begin{enumerate}
			\item[(i)] $X$ is isometric to an $N$-Euclidean   metric measure cone over an ${\sf RCD}(N-2,N-1)$ space;
			\item[(ii)] for $\mathcal L^1$-a.e.\ $t>0$, the super-level sets $\{u>t\}$ are metric balls  $ B_{r(t)}(x(t))$; moreover, $X$ is a cone with tip at $x(t)$.
		\end{enumerate}
		Actually, \eqref{eq:AVRPerEq} and the statement in (ii) can be improved to hold for \emph{every} $t\geq 0$. For $t\geq \max u$ the claim is trivial, so let us fix $t\in [0, \max u)$ and let $t_k\searrow t$ be such that  \eqref{eq:AVRPerEq} holds for each $t_k$. By monotone convergence, it is easily seen that ${\sf m}(\{u>t_k\})\nearrow {\sf m}(\{u>t\})$ and that the corresponding characteristic functions converge in $L^1(X, {\sf m})$.
		The lower semi-continuity of the perimeter implies 
		\begin{align*}
			{\rm Per}(\{u>t\})&\leq \liminf_{k\to \infty} {\rm Per}(\{u>t_k\})= \liminf_{k\to \infty} N\omega_N^\frac{1}{N}{\sf AVR}_ {\sf m}^\frac{1}{N} {\sf m}(\{u>t_k\})^\frac{N-1}{N} \\
			&=N\omega_N^\frac{1}{N}{\sf AVR}_ {\sf m}^\frac{1}{N} {\sf m}(\{u>t\})^\frac{N-1}{N},
		\end{align*}
		proving the claim, as the converse inequality always holds. 

From \eqref{eq:Omegau>0}, we obtain that $\Omega=\{u>0\}$ is a metric ball $B_r(\bar x)$ centred at a tip $\bar{x}$ for the cone structure of $X$; moreover, $r={\sf AVR}_ {\sf m}^{-{1}/{N}}R,$ where ${\sf m}(\Omega)=\omega_NR^N.$ 
\smallskip

\textbf{Step 2}.  \textit{If equality is attained in \eqref{inequality}, then the eigenvalue $\Lambda_ {\sf m}(\Omega)$ is simple, i.e., the corresponding eigenspace is $1$-dimensional.}

Let $u_1, u_2: \Omega\to \mathbb{R}$ be eigenfunctions relative to  $\Lambda_{\sf m}(\Omega)$. From step 1, up to a multiplicative constant, we can assume that both functions are positive in $\Omega$. In particular, both $$\int_\Omega u_1\, {\rm d}{\sf m}>0,  \int_\Omega u_2\, {\rm d}{\sf m}>0.$$ Thus there exists $c>0$ such that 
\begin{equation}\label{eq:intu1=u2}
\int_\Omega u_1\, {\rm d}{\sf m}= c\; \int_\Omega u_2\, {\rm d}{\sf m}.
\end{equation}
Define $\bar{u}:= u_1 - c u_2$. Since $\Delta$ is linear (here we used that ${\sf RCD}$ spaces are infinitesimally Hilbertian), it is easily seen that either $\bar{u}$ is an eigenfunction for $\Lambda_{\sf m}(\Omega)$, or it vanishes identically. Assume the former holds. Then, from step 1, we infer that $\bar u$ has a fixed sign and thus $\int_{\Omega} \bar u \, {\rm d}{\sf m}\neq 0$, contradicting  \eqref{eq:intu1=u2}.
We conclude that $\bar u \equiv 0$ on $\Omega$, and thus $u_1$ and $u_2$ are linearly dependent. 
\smallskip

		\textbf{Step 3}. {\it Explicit shape of the extremizer(s).}
  
  On the one hand, step 1 guarantees that $X$ is an $N$-Euclidean cone and $\Omega=B_r(\bar x)$ is a metric ball centered at a tip $\bar{x}$ of the cone structure of $X$, with $r={\sf AVR}_ {\sf m}^{-{1}/{N}}R,$ where ${\sf m}(\Omega)=\omega_NR^N.$ In this particular setting, it turns out by step 1 from \S\ref{subsection-sharp}  that the function $u$ defined in \eqref{extremal-candidate}   is an extremizer in \eqref{inequality}. On the other hand, by step 2, extremizers form a 1-dimensional eigenspace, which implies that, up to a multiplicative constant, every eigenfunction realizing equality in \eqref{inequality} has the form $$u(x)=U\left({\sf AVR}_ {\sf m}^{\frac{1}{N}}{\sf d}(\overline x,x)\right),\   x\in B_{{\sf AVR}_ {\sf m}^{-{1}/{N}}R}(\overline x)\simeq \Omega, $$ with $U$ as in  \eqref{U-definition}.   
	 \hfill $\square$

 \section{Stability in \eqref{inequality}: proof of Theorem \ref{theorem-main-stability}}
	
	The following lemma will be useful in the proof of Theorem \ref{theorem-main-stability}.
	
	\begin{lemma}[Upper semi-continuity of ${\sf AVR}$ under pmGH convergence] \label{lem:USC-AVR} Fix $N\in (1,\infty)$ and let $\left((X_j, {\sf d}_j, {\sf m}_j, \bar{x}_j)\right)_{j\in \mathbb{N}}$ be a sequence of pointed  ${\sf CD}(0,N)$ metric measure  spaces, converging in the pointed measured Gromov--Hausdorff sense to $(Y, {\sf d}_{Y}, {\sf m}_{Y}, y_{0})$. 
	
	Then  $(Y, {\sf d}_{Y}, {\sf m}_{Y})$ is a ${\sf CD}(0,N)$ metric measure  space and 
	\begin{equation}\label{eq:USC-AVR}
	{\sf AVR}_{{\sf m}_Y} \geq \limsup_{j\to \infty} {\sf AVR}_{{\sf m}_j}. 
	\end{equation}
	\end{lemma}

	\begin{proof} 
	The fact that the limit space $(Y, {\sf d}_{Y}, {\sf m}_{Y})$ is ${\sf CD}(0,N)$ follows from the stability of the ${\sf CD}(0,N)$ condition under pmGH convergence (see \cite[Theorem 29.25]{Villani}, after \cite{Sturm-2, LV}).    
	By the very definition \eqref{eq:defAVR} of ${\sf AVR}$, for every $\varepsilon>0$ there exists $\bar{R}>0$ such that
	\begin{equation}\label{eq:RAVRY}
	\frac{{\sf m}_Y(B_{\bar R}(y_0))}{\omega_N {\bar R}^N}\leq  {\sf AVR}_{{\sf m}_Y}  +\varepsilon.
	\end{equation}
	Combining the pmGH convergence with the ${\sf CD}(0,N)$ condition, it is readily checked (see for instance \cite[Lemma 4.13]{MV-CVPDE}) that 
	\begin{equation}\label{eq:convVolBallsAVR}
		{\sf m}_{j}(B_{\bar R}(\bar x_{j}))\to {\sf m}_{Y}(B_{\bar R}(y_{0})), \quad \text{as } j\to \infty.
		\end{equation}
	The combination of  \eqref{eq:RAVRY} with \eqref{eq:convVolBallsAVR} gives that there exists $j_0>0$ such that
	\begin{equation}\label{eq:mjleq}
		\frac{{\sf m}_{j}(B_{\bar R}(\bar x_{j}))}{ \omega_N {\bar R}^N }\leq  {\sf AVR}_{{\sf m}_Y} + 2 \varepsilon, \quad \text{for all } j\geq j_0.
		\end{equation}
	Since, by Bishop--Gromov theorem, the function $r\mapsto \frac{{\sf m}_{j}(B_r(\bar x_{j}))}{ \omega_N r^N }$ is monotone decreasing, 	the inequality \eqref{eq:mjleq} gives that
	$$
	{\sf AVR}_{{\sf m}_j}\leq   {\sf AVR}_{{\sf m}_Y} + 2 \varepsilon, \quad \text{for all } j\geq j_0.
	$$
	The claim \eqref{eq:USC-AVR} follows.
	\end{proof}

	\textbf{Proof of  Theorem \ref{theorem-main-stability}.}
 \smallskip

 \textbf{Step 1.} \textit{Stability for the shape of an almost optimal space.}
\\	Assume by contradiction that Theorem \ref{theorem-main-stability}/(i) does not hold. Then there exists $\eta_{0}>0$ and  a sequence $(X_{j}, {\sf d}_{j}, {\sf m}_{j})$ of ${\sf RCD}(0,N)$ spaces having asymptotic volume ratio equal to $\alpha>0$, admitting metric balls $B_{r_{j}}(\bar x_{j})$, with ${\sf m} (B_{r_j}(\bar x_j))=V$,  and functions $u_{j} \in H^{2,2}_{0} (B_{r}(\bar x_{j}))\setminus \{0\}$ such that 
	\begin{equation}\label{eq:ASSMdeltaj}
		\frac{\displaystyle\int_{B_{r_{j}}(\bar x_{j})} (\Delta u_{j})^2 {\rm d}{\sf m}_{j}}{\displaystyle\int_{B_{r_{j}}(\bar x_{j})} u_{j}^2 \, {\rm d}{\sf m}_{j}} \leq \alpha^\frac{4}{N} h_\nu^4 \left(\frac{\omega_N}{V} \right)^\frac{4}{N} + \delta_{j}, \quad \text{ with } \delta_{j}\searrow 0, 
		\end{equation}
		but 
		\begin{equation}\label{eq:Thesiseta0}
		 {\sf d}_{\rm pmGH}\left( (X_j, {\sf d}_{X_j}, {\sf m}_{X_j}, \bar x_j), (Y, {\sf d}_{Y}, {\sf m}_{Y}, y_{0})  \right)\geq \eta_{0},
		\end{equation}
		for any $N$-Euclidean metric measure cone $(Y, {\sf d}_{Y}, {\sf m}_{Y})$ over an ${\sf RCD}(N-2,N-1)$ space, where  $y_{0}\in Y$. Up to multiplying each $u_j$ by the normalizing constant $1/\|u_j\|_{L^2(X_j, {\sf m}_j)}$, we can assume that 
		\begin{equation}\label{eq:ujnormalized}
		\int_{B_{r_{j}}(\bar x_{j})} u_{j}^2 \, {\rm d}{\sf m}_{j} =1, \quad \text{for all }j\in \mathbb{N}.
		\end{equation}
		Notice that \eqref{eq:unablauL2}, which can be extended to any $H^{2,2}_{0}$ function by density, combined with \eqref{eq:ASSMdeltaj} and \eqref{eq:ujnormalized} gives that
		\begin{equation}\label{eq:ujBDDW12}
		\sup_{j\in \mathbb{N}} \int_{B_{r_{j}}(\bar x_{j})} \left( |\Delta u_j|^2 +  |\nabla u_{j}|^2 \right) \, {\rm d}{\sf m}_{j} <\infty. 	
		\end{equation}
		
		By Gromov pre-compactness theorem,  up to a subsequence,    $(X_{j}, {\sf d}_{j}, {\sf m}_{j}, \bar{x}_{j})$ converge in pmGH sense  to a pointed  metric measure space $(Y, {\sf d}_{Y}, {\sf m}_{Y}, y_{0})$. By the stability of the ${\sf RCD}(0,N)$ condition under pmGH convergence (see \cite{GMS}, after \cite{Sturm-2, LV,  Villani,  AGS2}) the limit space $(Y, {\sf d}_{Y}, {\sf m}_{Y}, y_{0})$ is ${\sf RCD}(0,N)$ as well.  Lemma \ref{lem:USC-AVR} yields that
		\begin{equation}\label{eq:USCAVR}
		{\sf AVR}_{{\sf m}_Y}\geq \alpha.
		\end{equation}
		
		The assumptions that ${\sf AVR}_{{\sf m}_j}=\alpha$,  ${\sf m}_j(B_1(\bar x_j))\leq v_0$,  ${\sf m}_j(B_{r_j}(\bar x_j))=V$ combined with Bishop-Gromov theorem yield a uniform bound on the radii $r_j$:
		\begin{equation*}
		\min\left\{1,\; \frac{V}{v_0} \right\} ^{1/N} \leq r_j \leq  \left( \frac{V}{\omega_N \, \alpha} \right)^{1/N}.
		\end{equation*}
		Thus, up to a subsequence,
		\begin{equation}\label{eq:RjtoR}
		r_j \to r, \quad \text{as } j\to \infty, \quad \min\left\{1,\; \frac{V}{v_0} \right\} ^{1/N} \leq r \leq  \left( \frac{V}{\omega_N \, \alpha} \right)^{1/N}.
		\end{equation}
		By \cite[Lemma 4.13]{MV-CVPDE}, the convergence \eqref{eq:RjtoR} implies that 
		\begin{equation}\label{eq:convVolBalls}
		{\sf m}_{Y}(B_{r}(y_{0}))=V.
		\end{equation}
		Arguing by density of $\mathsf{Test}_c(\Omega)$ in $H^{2,2}_0(B_{r_j}(\bar x_j))$  one can check that, for a function in  $H^{2,2}_0(B_{r_j}(\bar x_j))$, the local Laplacian on $B_{r_j}(\bar x_j)$ (obtained via integration by parts in $B_{r_j}(\bar x_j)$) coincides with the restriction of the ambient Laplacian  (obtained via integration by parts in $X_j$). Thus, recalling \eqref{eq:ujBDDW12},    we can apply   \cite[Theorem 4.4] {Ambrosio-Honda}  and obtain that there exists $u\in H^{2,2}_{0}(B_{r}(y_{0}))$ such that, up to a subsequence, $u_j\to u$ in the  $H^{1,2}$-strong sense, in particular
		\begin{equation}\label{eq:unormalized}
		\int_{B_{r}(y_0)} u^2 \, {\rm d}{\sf m}_{Y}= \lim_{j\to \infty} \int_{B_{r_{j}}(\bar x_{j})} u_{j}^2 \, {\rm d}{\sf m}_{j} \overset{\eqref{eq:ujnormalized}} {=}1,
		\end{equation}
		 and such that $\Delta u_{j}\rightharpoonup \Delta u$ in the  $L^2$-weak sense, giving
		 \begin{equation}\label{eq:LSCj}
			\int_{B_{r}(y_{0})} (\Delta u)^2 {\rm d}{\sf m}_{Y}  \leq \liminf_{j\to \infty}  \int_{B_{r_{j}}(\bar x_{j})} (\Delta u_{j})^2 {\rm d}{\sf m}_{j}.
				\end{equation}
				Combining  \eqref{eq:ASSMdeltaj},  \eqref{eq:USCAVR}, \eqref{eq:convVolBalls}, \eqref{eq:unormalized} and  \eqref{eq:LSCj}, we obtain that
				\begin{equation}\label{eq:limitleq}
			\frac{\displaystyle\int_{B_{r}(y_{0})} (\Delta u)^2 {\rm d}{\sf m}_{Y}}{\displaystyle\int_{B_{r}(y_{0})} u^2 {\rm d}{\sf m}_{Y} } \leq {\sf AVR}_{{\sf m}_{Y}}^\frac{4}{N} h_\nu^4 \left(\frac{\omega_N}{{\sf m}_{Y}(B_{r}(y_{0})} \right)^\frac{4}{N}.
							\end{equation}
		By the very definitions, \eqref{eq:limitleq} yields that $B_{r}(y_{0})\subset Y$ achieves the equality in  \eqref{inequality}. The rigidity proved in Theorem \ref{theorem-main-sharpness-rigidity}/(ii)  implies that $(Y, {\sf d}_{Y}, {\sf m}_{Y})$ is an $N$-Euclidean metric measure cone over an ${\sf RCD}(N-2,N-1)$ space, contradicting  \eqref{eq:Thesiseta0}. 
  
  The proof of \eqref{eq:r-Rleqeta} can be performed via an analogous argument by contradiction, after recalling the convergence of the radii $r_j\to r$ as in \eqref{eq:RjtoR} and the fact that the limit domain $B_{r}(y_{0})\subset Y$ must be a ball of radius $r=\alpha^{-1/N}R$, with  $R=(V/\omega_N)^{1/N}$, thanks to Theorem \ref{theorem-main-sharpness-rigidity}/(ii).
\smallskip

 \textbf{Step 2.} \textit{Stability for the shape of an almost optimal function.}
 \\ Assume by contradiction that Theorem \ref{theorem-main-stability}/(ii) does not hold. Then there exists $\eta_{0}>0$,  a sequence of ${\sf RCD}(0,N)$ spaces $(X_j, {\sf d}_j, {\sf m}_j)$ admitting metric balls $B_{r_{j}}(\bar x_{j})$, with ${\sf m}_j (B_{r_j}(\bar x_j))=V$,  and functions $u_{j} \in H^{2,2}_{0} (B_{r_j}(\bar x_{j}))\setminus \{0\}$ as in step 1,  such that \eqref{eq:ASSMdeltaj} holds and, for any  sequence $(c_j)_{j\in \mathbb{N}} \subset \mathbb{R}$, 
  \begin{equation}\label{eq:Thesiseta0function}
		 \frac{\|u_j- c_j \, \bar u_j\|_{H^{1,2}(X_j, {\sf d}_j, {\sf m}_j)}} {\|u_j\|_{L^2(X_j,  {\sf m}_j)}}\geq \eta_{0}, \quad \text{ for all } j\in \mathbb{N}, 
		\end{equation}
  where $u_j$ is extended to the value $0$ outside of $B_{r_j}(\bar x_{j})$, and $\bar u_j$ is defined by 
  \begin{equation*}
\bar u_j(x)=u^*\left(\alpha^{\frac{1}{N}}{\sf d}_j(\bar x_j,x)\right),\; \text{for all }x\in B_{\alpha^{-{1}/{N}}R}(\bar x_j), \quad \bar u(x)=0 \quad  \text{otherwise},
 \end{equation*}
 $u^*$ being given in \eqref{eq:defu*Intro} and  $R= (V / \omega_N)^{1/N}$. 
 \\ Up to multiplying each $u_j$ by the normalizing constant $1/\|u_j\|_{L^2(X_j, {\sf m}_j)}$, we can assume that \eqref{eq:ujnormalized} holds. Repeating verbatim step 1, we get that there exists $(Y, {\sf d}_{Y}, {\sf m}_{Y})$,  an $N$-Euclidean metric measure cone over an ${\sf RCD}(N-2,N-1)$ space, a metric ball $B_r(y_0)\subset Y$, and $u\in H^{2,2}_0(B_r(y_0))$  satisfying \eqref{eq:limitleq} and such that $u_j\to u$ in the $H^{1,2}$-strong sense.

 By the rigidity Theorem \ref{theorem-main-sharpness-rigidity}/(ii), we infer that there exists $c\in \mathbb{R}$ such that $u= c\,  \bar u$,  where
\begin{equation*}
\bar u(y)=  u^*\left(\alpha^{\frac{1}{N}}{\sf d}_Y(y_0,y)\right),\; \text{for all }y\in B_{\alpha^{-{1}/{N}}R}(y_0), \quad \bar u(y)=0 \quad  \text{otherwise}.
 \end{equation*}
 Thanks to Nobili and Violo \cite[Lemma 7.2]{NV-Adv}, we know that $\bar u_j\to \bar u$ in the $L^{2}$-strong sense and thus 
 \begin{equation}\label{eq:ujtocbarujL2}
 \| u_j- c\, \bar u_j\|_{L^2(X_j, {\sf m}_j)} \to 0, \quad \text{ as } j\to \infty.
 \end{equation}
 We finally estimate the gradient norms. From the chain rule for weak gradients and the eikonal equation \eqref{eikonal}, i.e.,  $|\nabla {\sf d}_Y(y_0, \cdot)|=1,$ ${\sf m}_Y$-a.e., we get that 
 $|\nabla \bar u|= \alpha^{\frac{1}{N}} |(u^*)'| \circ (\alpha^{\frac{1}{N}}{\sf d}_Y(y_0, \cdot))$ on $B_{\alpha^{-{1}/{N}}R}(y_0)$ and  equal to $0$ elsewhere. Analogously, we have that  $|\nabla \bar u_j|= \alpha^{\frac{1}{N}} |(u^*)'| \circ (\alpha^{\frac{1}{N}}{\sf d}_j(\bar{x}_j, \cdot))$ on $B_{\alpha^{-{1}/{N}}R}(\bar{x}_j)$ and  equal to $0$ elsewhere.
Applying again Nobili and Violo \cite[Lemma 7.2]{NV-Adv}, we obtain that $|\nabla \bar u_j|\to |\nabla \bar u|$ in the $L^{2}$-strong sense. This means that the convergence of $ \bar u_j$ to $\bar u$ is $H^{1,2}$-strong. Moreover, also the convergence of $u_j$ to $u=c\, \bar u$ is $H^{1,2}$-strong. It follows that
  \begin{equation}\label{eq:ujtocbarujW12}
 \| u_j- c\, \bar u_j\|_{H^{1,2}(X_j, {\sf d}_j,  {\sf m}_j)} \to 0, \quad \text{ as } j\to \infty,
 \end{equation}
 contradicting \eqref{eq:Thesiseta0function}.
  \hfill$\Box$

	\section{Examples} 	
	
	In this section we provide some examples which support the applicability of our results. We start with an example from the smooth setting. 
	
	\begin{example}\rm Let $n\in \{2,3\}$ and $f:[0,\infty)\to [0,1]$ be a smooth non-increasing function with $f(0)=1$ and $\lim_{s\to  \infty}f(s)=a\in (0,1].$ The rotationally invariant metric on $\mathbb R^n$ is given by 
		\begin{equation}\label{g-metric}
					g={\rm d}r^2+F^2(r){\rm d}\theta^2,
		\end{equation}
	where $F(r)=\ds\int_0^r f(s){\rm d}s$; here, ${\rm d}\theta^2$ stands for the standard metric on the unit $(n-1)$-dimensional sphere $\mathbb S^{n-1}\subset \mathbb R^n$. According to Carron \cite{Carron}, the Ricci curvature of the Riemannian manifold $(\mathbb R^n,g)$ is non-negative. By using twice the monotone L'H\^ospital rule -- as an alternative proof of Bishop--Gromov comparison principle -- one has that the function 
	$$r\mapsto \frac{\ds\int_0^r F^{n-1}(s){\rm d}s}{r^n}$$ is non-increasing on $[0,\infty)$; in particular, due to Balogh and Krist\'aly \cite{BK} one has that 
	$${\sf AVR}_{{\sf m}_g}=\lim_{R\to \infty}\frac{n\ds\int_0^R F^{n-1}(s){\rm d}s}{R^n}=a^{n-1}>0.$$ Accordingly, Theorem \ref{theorem-main} can be applied, obtaining the main inequality  \eqref{inequality}. 
	
	Moreover, if $\rho(x)={\sf d}_g(0,x)$, where ${\sf d}_g$ is the metric function inherited from the Riemannian manifold $(\mathbb R^n,g)$, it turns out that the Laplace--Beltrami operator for $\rho$ is 
	$$\Delta_g \rho=(n-1)\frac{f(\rho)}{F(\rho)}\geq 0.$$
	Therefore, one can apply  Theorems \ref{theorem-main-sharpness-rigidity} and \ref{theorem-main-stability}. 
	In particular, if equality holds in  \eqref{inequality} for some open bounded set $\Omega\subset \mathbb R^n$,   then we necessarily have that
	$${\sf AVR}_{{\sf m}_g}=\lim_{R\to 0}\frac{n\ds\int_0^R F^{n-1}(s){\rm d}s}{R^n}=1,$$
	i.e.,   $a=1$, which implies that $f\equiv 1$ and the metric $g$ from \eqref{g-metric} turns out to be the canonical metric in $\mathbb R^n$, $n\in \{2,3\}$, which is precisely the setting of Ashbaugh and  Benguria \cite{A-B-Duke}.  
	\end{example}
	
	We now present an example on cones in Euclidean spaces. 
	
	\begin{example}\rm Let $n\geq 1$ be an integer,   $\Sigma\subseteq \mathbb R^n$ be an open convex cone with vertex at the origin and let $w$ be a continuous function in $\overline \Sigma$, positive in $\Sigma$, and  positively homogeneous of degree $\alpha\geq 0$ such that the function $w^{1/\alpha}$ is concave in $\Sigma$ when $\alpha>0$ (if $\alpha=0$, we consider $w$ to be constant). If ${\sf m}=w \mathcal L^n$, then the above concavity property implies that $(\Sigma, {\sf d}_{\rm eu},{\sf m})$ is a ${\sf CD}(0,N)$ space with $N=n+\alpha$, see Villani \cite{Villani} and Cabr\'e, Ros-Oton and Serra \cite[Remark 1.4]{CRS}; in fact, a standard argument also shows that $(\Sigma, {\sf d}_{\rm eu},{\sf m})$ is an	${\sf RCD}(0,N)$ space. Various examples for cones $\Sigma$ and functions $w$ which verify the above properties can be found in Cabr\'e, Ros-Oton and Serra \cite[\S 2]{CRS} and Balogh, Guti\'errez and Krist\'aly \cite[\S 4]{BGK}.

		By the homogeneity of $w$, one has that
		$${\sf AVR}_{\sf m}=\frac{\ds\int_{B^\Sigma_1(0)}w{\rm d}\mathcal L^n }{\omega_{N}}>0,$$
		where $B^\Sigma_1(0)$ is the unit ball in $\Sigma$ with center at $0\in \overline \Sigma$. If $N=n+\alpha$ is close enough to 2 or 3 in the sense of  Theorem \ref{theorem-main}, we have the  inequality  \eqref{inequality}. 
		
		Moreover, since $\Delta u=w^{-1}{\rm div}(w \nabla u)$ on $\Sigma$, one has for $\rho(x)=|x|={\sf d}_{\rm eu}(0,x)$ that $\Delta\rho=\frac{N-1}{\rho}\geq 0$ on $\Sigma$; thus, we may apply the sharpness and rigidity results from Theorem \ref{theorem-main-sharpness-rigidity}.  
	\end{example}
		
		Another example is provided by  Euclidean cones in the sense of Bacher and
		Sturm \cite{Bacher-Sturm} and Ketterer \cite{Ketterer}, which already appeared in the characterization of the equality case in the isoperimetric  inequality \eqref{eqn-isoperimetric-2}, see \S \ref{subsection-2-3}.

	\begin{example}\rm Let $N>1$ and consider a compact metric measure space $(Z,{\sf d}_Z,{\sf m}_Z)$ with diameter not greater than $\pi,$ verifying the 
	${\sf RCD}(N-2,N-1)$ condition. Let $(C(Z),{\sf d}_c,{\sf m}_c)$  be the metric measure cone over $Z$, where $C(Z)=Z \times [0,\infty)/(Z \times \{0\})$, ${\sf m}_c=t^{N-1}{\rm d}t\otimes {\sf m}_Z$ and ${\sf d}_c$ is the cone metric given by $${\sf d}_c((x,s),(y,t))=\sqrt{s^2+t^2-2st\cos{\sf d}_Z(x,y)},\ \ (x,s),(y,t)\in C(z).$$
		One can prove that $(C(Z),{\sf d}_c,{\sf m}_c)$ verifies the 
		${\sf RCD}(0,N)$ condition and $${\sf AVR}_{{\sf m}_c}=\frac{{\sf m}_Z(Z)}{N\omega_N}>0.$$
 Moreover, $\Delta\rho=\frac{N-1}{\rho}\geq 0$ on $C(Z)$, where $\rho={\sf d}_c(\cdot,(\overline y,\overline t))$, the point $(\overline y,\overline t)$ being a tip of $C(Z)$. Therefore, our results can be applied for the required range of $N$.

	\end{example}


\begin{thebibliography}{99}
		
		\bibitem{Ambrosio2002} L. Ambrosio, Fine properties of sets of finite perimeter in doubling metric measure
		spaces. \textit{Set-Valued Var. Anal.} 10 (2002),  111--128.
		
		\bibitem{ABS-GAFA} L. Ambrosio, E.\ Bru\'e, D. Semola, Rigidity of the 1-Bakry–\'Emery inequality and sets of finite perimeter in RCD spaces. \textit{Geom. Funct. Anal.} 29, (2019), 949--1001. 
		
		\bibitem{ADMG} L. Ambrosio, S. Di Marino, N. Gigli, Perimeter as relaxed Minkowski content in metric measure spaces. \textit{Nonlinear
			Analysis} 153 (2017),  78--88.
		
		\bibitem{AGMR}
		L. Ambrosio; N. Gigli, A. Mondino, T. Rajala,
		Riemannian Ricci curvature lower bounds in metric measure spaces with $\sigma$-finite measure.
		\textit{Trans. Amer. Math. Soc.}  367 (2015), no. 7, 4661--4701.
		
		
		\bibitem{AGS} L. Ambrosio, N. Gigli, G. Savar\'e, Calculus and heat flow in metric measure spaces and applications to spaces
		with Ricci bounds from below. \textit{Invent. Math.} 195 (2014), no. 2, 289--391.
		
		\bibitem{AGS2} L. Ambrosio, N. Gigli, G. Savar\'e, Metric measure spaces with Riemannian Ricci curvature bounded from below.  \textit{Duke Math. J.} 163 (2014), no. 7, 1405-1490.
		
		\bibitem{Ambrosio-Honda} L. Ambrosio, S. Honda, Local spectral convergence in ${\sf RCD}^{*}(K,N)$ spaces, \textit{Nonlinear Analysis} 177 (2018) 1--23.
		
		\bibitem{AMS} L. Ambrosio, A. Mondino, G. Savaré,  On the Bakry-Émery condition, the gradient estimates and the local-to-global property of ${\sf RCD}^*(K,N)$ metric measure spaces. \textit{J. Geom. Anal.} 26 (2016), no. 1, 24--56.
		
		\bibitem{APPS} G. Antonelli, E. Pasqualetto,  M. Pozzetta, D. Semola, Asymptotic isoperimetry on non collapsed spaces with lower Ricci bounds.\textit{ Math. Ann.} 389 (2024), no. 2, 1677--1730.
		
		\bibitem{APPV} G. Antonelli, E. Pasqualetto, M. Pozzetta, I.Y. Violo, Topological regularity of isoperimetric sets
		in PI spaces having a deformation property. \textit{Proc. R. Soc. Edinb. Sect. A Math.} (2023), https://
		doi.org/10.1017/prm.2023.105, published online; arXiv:2303.01280.
		
		
		
		
		\bibitem{A-B-Duke} M. Ashbaugh, R. Benguria, On Rayleigh's conjecture for the clamped plate and its generalization to three dimensions. \textit{Duke Math.\ J.} 78 (1) (1995), 1--17. 
		
		\bibitem{A-L-Pisa}
		M. Ashbaugh, R.S. Laugesen, Fundamental tones and buckling loads of clamped plates. \textit{Ann. Sc. Norm. Super. Pisa}, Cl. Sci. (4) 23 (2) (1996), 383--402. 
		
		\bibitem{Bacher-Sturm} K. Bacher, K.-T. Sturm,  Ricci Bounds for Euclidean and Spherical Cones Singular Phenomena and
		Scaling in Mathematical Models, Springer, Cham (2014), pp. 3--23. 
		
		\bibitem{BGK} Z.M. Balogh, C.E. Gutiérrez, A. Kristály,  Sobolev inequalities with jointly concave weights on convex cones. \textit{Proc. Lond. Math. Soc.} (3) 122 (2021), no. 4, 537--568.
		
		
		\bibitem{BK} Z.\ Balogh,   A.\ Krist\'{a}ly,  Sharp isoperimetric and Sobolev inequalities in spaces with nonnegative Ricci curvature. \textit{Math. Ann.}  385 (2023), no. 3-4, 1747--1773. 
		
	 \bibitem{Barbosa-Kristaly}	E. Barbosa, A. Kristály, 
		Second-order Sobolev inequalities on a class of Riemannian manifolds with nonnegative Ricci curvature. 
		\textit{Bull. Lond. Math. Soc.} 50 (2018), no. 1, 35--45. 
		
		\bibitem{BPS}	E. Bru\'e, E. Pasqualetto, D. Semola,  Rectifiability of the reduced boundary for sets of finite
		perimeter over ${\sf RCD}(K, N)$  spaces. \textit{J. Eur. Math. Soc.} 25 (2023), 413--465. 
		
		\bibitem{BPS-ASNS}	E. Bru\'e, E. Pasqualetto, D. Semola, Constancy of the dimension in codimension one and locality of the unit normal on ${\sf RCD}(K,N)$ spaces. \textit{Ann. Sc. Norm. Super. Pisa Cl. Sci. (5)} XXIV (2023), 1765--1816.
		
		\bibitem{CRS} X. Cabré, X. Ros-Oton, J. Serra, Sharp isoperimetric inequalities via the ABP method. \textit{J. Eur. Math. Soc.} 18 (2016), 2971--2998.  
		
		\bibitem{Carron} G. 
		Carron,  Euclidean volume growth for complete Riemannian manifolds. \textit{Milan J. Math.} 88 (2020), 455--478.  
		
		\bibitem{C-M} F. Cavalletti, D. Manini, Rigidities of isoperimetric inequality under nonnegative Ricci curvature,
		to appear, \textit{J. Eur. Math. Soc.} (2024).
  
  \bibitem{C-M-Inv} F. Cavalletti, F., E. Milman, The globalization theorem for the Curvature-Dimension condition. \textit{Invent. math.} 226, 1–137 (2021).  
  
		
		\bibitem{Chasman-Langford} L.M. Chasman, J.J. Langford, The clamped plate in Gauss space.
		\textit{Ann. Mat. Pura Appl.} (4), (6)195 (2016), 1977--2005.
		
		\bibitem{Cheeger}	J. Cheeger, Differentiability of Lipschitz functions on metric measure spaces. \textit{Geom. Funct. Anal.} 9 (1999),
		428--517.
		
		\bibitem{CDS} 	C.V. Coffman, R.J. Duffin, D.H. Shaffer, The fundamental mode of vibration of a clamped annular plate is not of one sign, in: Constructive Approaches to Mathematical Models (Proc. Conf. in Honor of R.J. Duffin), Pittsburgh, PA, 1978, Academic Press, New York-London-Toronto, Ont, 1979, pp. 267--277.
		
		\bibitem{Deng-Zhao} Q. Deng, X. Zhao,  
		Failure of strong unique continuation for harmonic functions on RCD spaces. 
	\textit{J. Reine Angew. Math.} 795 (2023), 221--241.
		
		\bibitem{DeP-G} G. De Philippis, N. Gigli,  From volume cone to metric cone in the nonsmooth setting. \textit{Geom. Funct. Anal.} 26 (2016), no. 6, 1526--1587.
		
		\bibitem{Duffin} R.J. Duffin, Nodal lines of a vibrating plate. \textit{J. Math. Phys.} 31 (1953), 294--
		299. 

  \bibitem{GGS} F. Gazzola, H.-C. Grunau, G. Sweers, Polyharmonic boundary value
problems. Positivity preserving and nonlinear higher order elliptic equations
in bounded domains. Lecture Notes in Mathematics, 1991. Springer, Berlin, 
2010. 
		
		\bibitem{Gigli1} N. Gigli, On the differential structure of metric measure spaces and applications. \textit{Mem. Amer. Math.} Soc. 236 (2015), no. 1113, vi+91 pp.
		
		\bibitem{Gigli2} N. Gigli, Lecture notes on differential calculus on RCD spaces. \textit{Publ. Res. Inst. Math. Sci.} 54 (2018), no. 4, 855--918.
		
		\bibitem{GMS} N. Gigli, A. Mondino, G. Savaré, Convergence of pointed non-compact metric measure spaces and stability of Ricci curvature bounds and heat flows. \textit{Proc. Lond. Math. Soc.} (3) 111 (2015), no. 5, 1071--1129. 
		
		\bibitem{Hebey} E. Hebey, Nonlinear analysis on manifolds: Sobolev spaces and inequalities. Courant Lecture Notes in Mathematics, 5.
		New York University, Courant Institute of Mathematical Sciences, New
		York; American Mathematical Society, Providence, RI, 1999.
		
		
		\bibitem{Ketterer} C. Ketterer,  Cones over metric measure spaces and the maximal diameter theorem. \textit{J. Math. Pures
		Appl.} (9) 103(5) (2015), 1228--1275. 
		
		\bibitem{Kristaly-Adv} A. Kristály,  Fundamental tones of clamped plates in nonpositively curved spaces. \textit{Adv. Math.} 367 (2020), 107113, 39 pp.
		
		\bibitem{Kristaly-GAFA} A. Kristály, Lord Rayleigh's conjecture for vibrating clamped plates in positively curved spaces. \textit{Geom. Funct. Anal.} 32 (2022), no. 4, 881--937.
		
		\bibitem{Jiang} R. Jiang, Lipschitz continuity of solutions of Poisson equations in metric measure space. \textit{Potential Anal.} 37 (2012), no. 3, 281--301.
		
		\bibitem{Leylekian-ARMA} R. Leylekian,  Towards the optimality of the ball for the Rayleigh conjecture concerning the clamped plate. \textit{Arch. Ration. Mech. Anal.} 248 (2024), no. 2, Paper No. 28, 35 pp.
		
		\bibitem{Leylekian-1} R. Leylekian,  Sufficient conditions yielding the Rayleigh Conjecture for the clamped plate. \textit{ Ann. Mat. Pura Appl.} (2024). https://doi.org/10.1007/s10231-024-01454-y
		
		\bibitem{LV}  J. Lott, C. Villani, Ricci curvature for metric measure spaces via optimal transport. \textit{Ann. of Math. (2)}  169 (3) (2009), 903--991.
		
		
		
		\bibitem{Miranda} M. Miranda, Jr., Functions of bounded variation on ``good'' metric spaces. \textit{J. Math. Pures Appl.} (9), 82
		(2003), pp. 975--1004.
		
		\bibitem{MN-JEMS} A. Mondino, A. Naber, Structure theory of metric measure spaces with lower Ricci curvature bounds. \textit{J. Eur. Math. Soc.} 21 (2019), no. 6,  1809--1854.
		
		\bibitem{MSemola} A. Mondino, D. Semola, 
		Polya-Szego inequality and Dirichlet $p$-spectral gap for non-smooth spaces with Ricci curvature bounded below. 
		\textit{J. Math. Pures Appl.} (9) 137 (2020), 238--274. 
		
		\bibitem{MV-CVPDE} A. Mondino, M. Vedovato,
		A Talenti-type comparison theorem for ${\sf RCD}(K, N)$ spaces and applications. \textit{Calc. Var.}
		157 (2021). https://doi.org/10.1007/s00526-021-01971-1.
		
		\bibitem{Nadirashvili}	N.S. Nadirashvili, Rayleigh's conjecture on the principal frequency of the clamped plate. \textit{Arch. Ration. Mech. Anal.} 129 (1) (1995) 1--10. 
		
		\bibitem{NV} F. Nobili, I.Y. Violo,  Rigidity and almost rigidity of Sobolev inequalities on compact spaces with lower Ricci curvature bounds,  \textit{Calc. Var. Partial Differential Equations},  61 (2022), no. 5, Paper No. 180, 65 pp.
		
		 \bibitem{NV-Adv} F. Nobili, I.Y. Violo, Stability of Sobolev inequalities on Riemannian manifolds with Ricci curvature lower bounds. \textit{Adv. Math.} 440 (2024), Paper No. 109521, 58 pp. 
		
		
		
		\bibitem{Olver-etal} F.W.J. Olver, D.W. Lozier, R.F. Boisvert, C.W. Clark (Eds.), NIST Handbook of Mathematical Functions, Cambridge University Press, Cambridge, 2010. 
		
		\bibitem{Rajala} T. Rajala, Local Poincar\'e inequalities from stable curvature conditions on metric spaces,  \textit{Calc. Var.} (2012) 44:477--494.
		
		\bibitem{Rayleigh} J.W.S. Rayleigh, The Theory of Sound. 2nd edition, revised and enlarged
		(in two volumes). Dover Publication, New York, (1945). Republication of the
		1894/1896 edition.
				
		
		\bibitem{Sturm-2} K.-T. Sturm, On the geometry of metric measure spaces. II. \textit{Acta
			Math.} 196 (1) (2006), 133--177.
		
		
		\bibitem{Szego} G. Szeg\H o, On membranes and plates. \textit{Proc. Natl. Acad. Sci. USA} 36 (1950) 210--216. 
		
		
		\bibitem{Talenti}
		G. Talenti, On the first eigenvalue of the clamped plate. \textit{Ann. Mat. Pura Appl.} (4) 129 (1981) 265--280. 
		
	\bibitem{Villani} C. 	Villani,  Optimal Transport. Old and New. Grundlehren der Mathematischen Wissenschaften, vol.
		338. Springer, Berlin, 2009. 
		
		\bibitem{Watson} G.N. Watson,  A treatise on the theory of Bessel functions. Reprint of the second (1944) edition. Cambridge Mathematical Library. Cambridge University Press, Cambridge, 1995.
	\end{thebibliography}
\end{document}